\documentclass[12pt]{amsart}
\usepackage{amsmath,amsthm,amsfonts,amscd,amssymb,eucal,latexsym,mathrsfs, appendix, yhmath, setspace}
\usepackage[numbers,sort&compress]{natbib}
\usepackage[all,cmtip]{xy}
\usepackage{enumerate}

\newtheorem{theorem}{Theorem}[section]

\newtheorem{lemma}[theorem]{Lemma}
\newtheorem{proposition}[theorem]{Proposition}

\theoremstyle{definition}

\newtheorem{remark}[theorem]{Remark}

\newtheorem{problem}[theorem]{Problem}

\newcommand{\htopol}{{\text{\rm h}}_{\text{\rm top}}}

\newcommand{\id}{{\rm id}}

\newcommand{\cG}{{\mathcal G}}

\newcommand{\cU}{{\mathcal U}}

\newcommand{\cX}{{\mathcal X}}

\newcommand{\Nb}{{\mathbb N}}

\newcommand{\Zb}{{\mathbb Z}}

\newcommand{\sM}{{\mathscr M}}

\newcommand{\IN}{{\rm IN}}

\allowdisplaybreaks

\begin{document}

\title[Null actions and RIM non-open extensions]{Null actions and RIM non-open extensions of strongly proximal actions}

\author{Hanfeng Li}
\author{Zhen Rong}

\address{\hskip-\parindent
H.L., Center of Mathematics, Chongqing University,
Chongqing 401331, China. \\
Department of Mathematics, SUNY at Buffalo,
Buffalo, NY 14260-2900, U.S.A.}
\email{hfli@math.buffalo.edu}

\address{\hskip-\parindent
Z.R., College of Mathematics and Statistics, Chongqing University,
Chongqing 401331, China.}
\email{20130601002@cqu.edu.cn}

\date{February 9, 2019}


\begin{abstract}
Answering a question of Glasner, we show that any finitely generated nonabelian free group has a minimal null action which is a
RIM non-open extension of an effective strongly proximal action.
\end{abstract}

\maketitle


\section{Introduction} \label{S-introduction}

In this work we consider continuous actions of a countably infinite discrete group $\Gamma$ on a compact metrizable space $X$. Such an action is called {\it tame} \cite{Kohler, Glasner06} if the induced $\Gamma$-action on the space $C(X)$ of all complex-valued continuous functions on $X$ contains no isomorphic dynamical copy of $\ell^1(\Gamma)$, i.e. for any $f\in C(X)$ there is no constant $C>0$ such that
$$ C\|g\|_1 \le \|\sum_{s\in \Gamma}g(s) (sf)\|_\infty\le C^{-1}\|g\|_1$$
for all $g\in \ell^1(\Gamma)$, where $(sf)(x)=f(s^{-1}x)$ for all $x\in X$.
Tameness can also be described in terms of the Ellis enveloping semigroup $E(X, \Gamma)$ of the action $\Gamma\curvearrowright X$, which is the closure of the image of $\Gamma$ in the product space $X^X$. Indeed, the following conditions are equivalent \cite{BFT, Glasner06, GM06, GMU}:
\begin{enumerate}
\item $\Gamma\curvearrowright X$ is tame,
\item $E(X, \Gamma)$ is a separate Fr\'{e}chet compact space, hence with cardinality at most $2^{\aleph_0}$,
\item $E(X, \Gamma)$ does not contain a homeomorphic copy of the Stone-\u{C}ech compactification of $\Nb$,
\item every element of $E(X, \Gamma)$ is a Baire class $1$ function from $X$ to itself.
\end{enumerate}
Actually the equivalence between (2) and (3) is a dynamical version of the Bourgain-Fremlin-Talagrand dichotomy theorem \cite{Glasner06, GM06}.

In \cite{Glasner18} Glasner established a structure theorem for tame minimal actions, extending the results for abelian $\Gamma$ in \cite{Glasner07, Huang, KL07}.
The action $\Gamma\curvearrowright X$ is called {\it minimal} if there is no proper nonempty closed $\Gamma$-invariant subset of $X$.
It is called {\it strongly proximal} if for the induced $\Gamma$-action on the compact space $\sM(X)$ of all Borel probability measures on $X$, every nonempty closed $\Gamma$-invariant subset intersects with $X$ \cite{Glasner75a}.
Given another continuous action of $\Gamma$ on a compact metrizable space $Y$, a $\Gamma$-equivariant continuous surjective map $\pi: X\rightarrow Y$ is called a {\it factor map} or {\it extension}. The extension $\pi$ is said to be {\it strongly proximal} if for every $y\in Y$ and every $\mu\in \sM(X)$ with support contained in $\pi^{-1}(y)$, the orbit closure of $\mu$ in $\sM(X)$ intersects with $X$ \cite{Glasner75a}.
A {\it relative invariant measure} (RIM) for $\pi$ is a continuous $\Gamma$-equivariant map from $Y$ to
$\sM(X)$ such that the support of the image of each $y\in Y$ is contained in $\pi^{-1}(y)$ \cite{Glasner75b}. The extension $\pi$ is called {\it point-distal} if there is a point $x\in X$ with dense orbit such that for any $x\neq x'\in X$ with $\pi(x)=\pi(x')$ the orbit closure of $(x, x')$ does not intersect with the diagonal \cite[VI.4.1]{Vries}.
It is called {\it almost one-to-one} if there is some $x\in X$ with dense orbit in $X$ such that $\pi^{-1}(\pi(x))=\{x\}$ \cite[IV.6.1]{Vries}.
 It is called {\it equicontinuous} or {\it isometric} if, given a compatible metric $\rho$ on $X$, for any $\varepsilon>0$ there exists $\delta>0$ such that for any $x, x'\in X$ with $\pi(x)=\pi(x')$ and $\rho(x, x')<\delta$ one has $\rho(sx, sx')<\varepsilon$ for all $s\in \Gamma$ \cite[V.2.1]{Vries}. Then Glasner's structure theorem states as follows:
for every tame minimal action $\Gamma\curvearrowright X$ there is a commutative diagram of minimal $\Gamma$-actions
\begin{equation*}
\xymatrix
{
& \tilde{X} \ar[dd]_\pi  \ar[dl]_\eta & X^* \ar[l]_-{\theta^*} \ar[d]^{\iota} \\
X & & Z \ar[d]^\sigma\\
& Y & Y^* \ar[l]^\theta
}
\end{equation*}
such that $\Gamma\curvearrowright \tilde{X}$ is tame, $\eta$ is a strongly proximal extension, $\Gamma\curvearrowright Y$ is a strongly proximal action, $\pi$ is a point-distal extension with a unique RIM, $\theta$, $\theta^*$ and $\iota$ are almost one-to-one extensions, and $\sigma$ is an isometric extension.

The map $\pi$ is open exactly when $\theta$ and $\theta^*$ are trivial. In such case the above diagram reduces to
\begin{equation*}
\xymatrix
{
& \tilde{X}  \ar[dl]_\eta  \ar[d]^{\iota} \ar@/^2pc/@{>}^\pi[dd]\\
X &  Z \ar[d]^\sigma\\
& Y
}
\end{equation*}
This leads Glasner to ask the following question \cite[Problem 5.5]{Glasner18}
\begin{problem} \label{P-open}
Let $\Gamma\curvearrowright \tilde{X}$ be a tame minimal action, and let $\Gamma\curvearrowright Y$ be a strongly proximal action. If $\pi:\tilde{X}\rightarrow Y$ is a RIM extension, then must $\pi$ be open?
\end{problem}

When $\Gamma$ is amenable, every minimal strongly proximal action is the trivial action on a singleton \cite[Theorem III.3.1]{Glasner76}. Thus Problem~\ref{P-open} has affirmative answer in this case, even without assuming $\Gamma\curvearrowright \tilde{X}$ to be tame.

Our goal in this paper is to answer Problem~\ref{P-open} negatively. In fact we shall construct counterexample for a weaker statement. Recall that for a sequence $\mathfrak{s}=\{s_n\}_{n\in \Nb}$ in $\Gamma$, the {\it sequence topological entropy} of an action $\Gamma\curvearrowright X$ with respect to $\mathfrak{s}$ and a finite open cover $\cU$ of $X$ is defined as
$$ \htopol(X, \cU; \mathfrak{s})=\limsup_{n\to \infty}\frac{1}{n}\log N(\bigvee_{i=1}^ns_i^{-1}\cU),$$
where $N(\bigvee_{i=1}^ns_i^{-1}\cU)$ denotes the minimal number of elements of $\bigvee_{i=1}^ns_i^{-1}\cU$ needed to cover $X$. The action $\Gamma\curvearrowright X$ is called {\it null} if $\htopol(X, \cU; \mathfrak{s})=0$ for all $\mathfrak{s}$ and $\cU$ \cite{Goodman, HLSY}. It is known that null actions are tame (see Section~\ref{S-indep}).
An action $\Gamma\curvearrowright Y$ is called {\it effective} if for any distinct $s, t$ in $\Gamma$ one has $sy\neq ty$ for some $y\in Y$. Our main result is

\begin{theorem} \label{T-main}
For any nonabelian finitely generated free group $\Gamma$, there are a null (hence tame) minimal action $\Gamma\curvearrowright \tilde{X}$, an effective strongly proximal action $\Gamma\curvearrowright Y$, and a point-distal non-open extension $\tilde{X}\rightarrow Y$ with a unique RIM.
\end{theorem}

There are two ingredients in our construction. McMahon \cite[Example 3.2.(1)]{McMahon78} constructed examples of RIM non-open extension $\Gamma\curvearrowright \tilde{X}$ of minimal equicontinuous actions $\Gamma\curvearrowright Y$ for $\Gamma=G\times \Zb$, where $G$ is any dense countable subgroup of the $p$-adic integer group $\Zb_p$. These examples do not provide counterexamples for Problem~\ref{P-open} for three reasons. The first is that $\Gamma$ in these examples is abelian. The second is that $\Gamma\curvearrowright Y$ is equicontinuous instead of strongly proximal. The third is that it is not clear whether $\Gamma\curvearrowright \tilde{X}$ in these examples are tame or not. Among these difficulties, the third one is most difficult. To prove Theorem~\ref{T-main} we modify McMahon's construction to handle these three difficulties. Our second ingredient is the combinatorial independence developed in \cite{KL07}. It enables us to turn the question of checking tameness or nullness to a combinatorial problem. The latter is still nontrivial but manageable.

This paper is organized as follows. We recall the basic of combinatorial independence in Section~\ref{S-indep}. McMahon's construction is recalled in Section~\ref{S-construction}. As a showcase of our technique, we construct some null minimal actions for every residually finite group in Section~\ref{S-RF}. Theorem~\ref{T-main} is proved in Section~\ref{S-free}.

Throughout this paper $\Gamma$ will be a countably infinite group with identity element $e_\Gamma$. All $\Gamma$-actions are assumed to be continuous actions on compact metrizable spaces unless specified otherwise. For each compact metrizable space $X$, we denote by $\sM(X)$ the space of all Borel probability measures on $X$, equipped with the weak${}^*$-topology.

\noindent{\it Acknowledgments.}
H.~L. was partially supported by NSF and NSFC grants. We are grateful to Eli Glasner for helpful comments.

\section{Combinatorial independence} \label{S-indep}

In this section we recall the combinatorial independence description of tameness and nullness \cite{KL07, KL}. Let $\Gamma$ act on a compact metrizable space $X$ continuously.

Let $(A_1, A_2)$ be a pair of subsets of $X$. We say that a set $M\subseteq \Gamma$ is an {\it independence set} for $(A_1, A_2)$ if the collection $\{(s^{-1}A_1, s^{-1}A_2): s\in M\}$ is independent in the sense that $\bigcap_{s\in F}s^{-1}A_{\omega(s)}\neq \emptyset$ for every nonempty finite set $F\subseteq M$ and $\omega\in \{1, 2\}^F$.

We say that a pair $(x_1, x_2)\in X^2$ is an {\it IT-pair} if for every product neighborhood $U_1\times U_2$ of $(x_1, x_2)$ the pair $(U_1, U_2)$ has an infinite independence set.

The following is the combinatorial independence characterization of tameness \cite[Proposition 6.4]{KL07} \cite[Proposition 8.14]{KL}. It's based on the proof of Rosenthal's $\ell_1$ theorem \cite{Rosenthal74, Rosenthal78}.

\begin{proposition} \label{P-tame}
The action $\Gamma\curvearrowright X$ is tame if and only if $X$ has no non-diagonal IT-pairs.
\end{proposition}

We say that a pair $(x_1, x_2)\in X^2$ is an {\it IN-pair} if for every product neighborhood $U_1\times U_2$ of $(x_1, x_2)$ the pair $(U_1, U_2)$ has arbitrarily large finite  independence sets. We denote by $\IN_2(X, \Gamma)$ the set of all IN-pairs.

The following summarizes the basic properties of IN-pairs and gives the combinatorial independence characterization of nullness \cite[Proposition 5.4]{KL07}.

\begin{proposition} \label{P-null}
The following are true:
\begin{enumerate}
\item Let $(A_1, A_2)$ be a pair of closed subsets of $X$ which has arbitrarily large finite independence sets. Then there exists an IN-pair $(x_1, x_2)$ with $x_j\in A_j$ for $j=1, 2$.

\item The action $\Gamma\curvearrowright X$ is null if and only if $X$ has no non-diagonal IN-pairs.

\item  $\IN_2(X, \Gamma)$ is a closed $\Gamma$-invariant subset of $X^2$.

\item Let $\pi: (X, \Gamma)\rightarrow (Y, \Gamma)$ be a factor map. Then $\pi^2(\IN_2(X, \Gamma))=\IN_2(Y, \Gamma)$.
\end{enumerate}
\end{proposition}

It follows from Propositions~\ref{P-tame} and \ref{P-null} that null actions are tame. We remark that there are minimal tame nonnull subshifts for $\Zb$ \cite[Section 11]{KL07}.

\section{McMahon's construction} \label{S-construction}

We recall McMahon's construction of RIM extensions in \cite{McMahon76, McMahon78}. See \cite[VI.6.5]{Vries} for details.

Let $\Gamma$ act on a compact metrizable space $X$ continuously and assume that the action is minimal and $X$ consists of more than one orbit. Let $x_1$ be a point of $X$ such that the stabilizer group $\{s\in \Gamma: sx_1=x_1\}$ of $x_1$ is trivial.
Also let  $f: X\setminus \{x_1\}\rightarrow \{1, -1\}$ be a continuous function such that it cannot be extended to $X$ continuously. Then $\Gamma$ has a minimal continuous action on some compact metrizable space $X_f$ with the following properties:
\begin{enumerate}
\item there is a continuous $\Gamma$-equivariant surjective map $\pi_f: X_f\rightarrow X$ such that $\pi_f^{-1}(x)$ consists of one point if $x\in X\setminus \Gamma x_1$ and two points if $x\in \Gamma x_1$,
\item the function $f\circ\pi_f$ on $X_f\setminus \pi_f^{-1}(x_1)$ extends to a continuous function $\tilde{f}: X_f\rightarrow \{1, -1\}$.
\end{enumerate}
Furthermore, as an extension of $\Gamma\curvearrowright X$, the action $\Gamma\curvearrowright X_f$ is unique up to conjugacy.
Indeed, $X_f$ is constructed as follows. Take a point $x_0\in X\setminus \Gamma x_1$ and denote by $Z$ the Stone-\u{C}ech compactification of $\Gamma x_0$ equipped with the relative topology from $\Gamma x_0\subseteq X$. Then the induced $\Gamma$-action on $Z$ is minimal and $f\circ \pi_Z|_{\Gamma x_0}$ extends uniquely to a continuous function $f_Z: Z\rightarrow \{1, -1\}$, where $\pi_Z: Z\rightarrow X$ is the factor map extending the embedding $\Gamma x_0\hookrightarrow X$. Define an equivalence relation on $Z$ by $z_1\sim z_2$ if $f_Z(sz_1)=f_Z(sz_2)$ for all $s\in \Gamma$ and $\pi_Z(z_1)=\pi_Z(z_2)$. This equivalent relation is $\Gamma$-invariant and closed. Then $X_f$ is defined as the quotient space $Z/\sim$. Though $Z$ is possibly not metrizable, $X_f$ is. Since $z_1\sim z_2$ implies that $\pi_Z(z_1)=\pi_Z(z_2)$, there is a unique continuous map $\pi_f: X_f\rightarrow X$ such that the composition map $Z\rightarrow X_f\overset{\pi_f}\rightarrow X$ is equal to $\pi_Z$. As both $Z\rightarrow X_f$ and $\pi_Z$ are $\Gamma$-equivariant, so is $\pi_f$.
Since $z_1\sim z_2$ implies that $f_Z(z_1)=f_Z(z_2)$, there is also a unique continuous function $\tilde{f}: X_f\rightarrow \{1, -1\}$ such that the composition function $Z\rightarrow X_f\overset{\tilde{f}}\rightarrow \{1, -1\}$ is equal to $f_Z$. Then $f\circ \pi_f=\tilde{f}$ on the image of $\Gamma x_0$ under the quotient map $Z\rightarrow X_f$.
As this image is dense in $X_f$ and both $f\circ \pi_f$ and the restriction of $\tilde{f}$ are continuous on $X_f\setminus \pi_f^{-1}(x_1)$, one has $f\circ \pi_f=\tilde{f}$ on $X_f\setminus \pi_f^{-1}(x_1)$. We remark that another way of defining $X_f$ is to let it be the Gelfand spectrum of the $\Gamma$-invariant $C^*$-subalgebra of $\ell^\infty(\Gamma)$ generated by functions of the form $t\mapsto g(tx_0)$ for $t\in \Gamma$ and $g\in C(X)$ or $g=f$, where $\ell^\infty(\Gamma)$ is equipped with the supremum norm and pointwise multiplication, addition and conjugation and $\Gamma$ acts on $\ell^\infty(\Gamma)$ via $(sh)(t)=h(s^{-1}t)$ for all $s, t\in \Gamma$ and $h\in \ell^\infty(X)$.

We take an excursion to discuss when $\Gamma\curvearrowright X_f$ is null.
Set $X_{f, +}=\tilde{f}^{-1}(1)$ and $X_{f, -}=\tilde{f}^{-1}(-1)$. Also set $X_+=f^{-1}(1)$ and $X_-=f^{-1}(-1)$.
The following result tells us how to characterize nullness of $\Gamma\curvearrowright X_f$ from information on $\Gamma\curvearrowright X$ and $f$.

\begin{proposition} \label{P-construction null}
The action $\Gamma\curvearrowright X_f$ is null if and only if $\Gamma\curvearrowright X$ is null and the pair $(X_+, X_-)$ does not have arbitrarily large finite independence sets.
\end{proposition}

Note that as the sets $X_+$ and $X_-$ are not closed, the condition ``not having arbitrarily large finite independence sets'' is not automatically satisfied if $\Gamma\curvearrowright X$ is null.

Proposition~\ref{P-construction null} follows from Lemmas~\ref{L-null1} and \ref{L-null} below.

\begin{lemma} \label{L-null1}
The action $\Gamma\curvearrowright X_f$ is null if and only if $\Gamma\curvearrowright X$ is null and the pair $(X_{f,+}, X_{f,-})$ does not have arbitrarily large finite independence sets.
\end{lemma}
\begin{proof}
Assume that $\Gamma\curvearrowright X_f$ is null. Since $X_{f, +}$ and $X_{f, -}$ are disjoint closed subsets of $X_f$, by (1) and (2) of Proposition~\ref{P-null} the pair $(X_{f,+}, X_{f,-})$ does not have arbitrarily large finite independence sets. Also from (2) and (4) of Proposition~\ref{P-null} we know that factors of null actions are null. Thus $\Gamma\curvearrowright X$ is null.
This proves the ``only if'' part.

Assume that $\Gamma\curvearrowright X_f$ is nonnull and $\Gamma\curvearrowright X$ is null. By (2) of Proposition~\ref{P-null}  there is an IN-pair $(z_1, z_2)$ in $X_f^2$ with $z_1\neq z_2$. By (4) of Proposition~\ref{P-null} the pair $(\pi_f(z_1), \pi_f(z_2))$ is an IN-pair in $X^2$.  Since $\Gamma\curvearrowright X$ is null, by (2) of  Proposition~\ref{P-null} the IN-pairs in $X^2$ must be diagonal ones. Thus $\pi_f(z_1)=\pi_f(z_2)$. Then $\pi_f(z_1)\in \Gamma x_1$. Say, $\pi_f(z_1)=sx_1$ for some $s\in \Gamma$. Then $\pi_f(s^{-1}z_1)=\pi_f(s^{-1}z_2)=x_1$, and hence $\{s^{-1}z_1, s^{-1}z_2\}=\pi_f^{-1}(x_1)$. By (3) of Proposition~\ref{P-null} the pair $(s^{-1}z_1, s^{-1}z_2)$ is also an IN-pair. Note that $X_{f, +}$ and $X_{f, -}$ are both closed and open subsets of $X_f$, and each of them contains exactly one of $ s^{-1}z_1$ and $s^{-1}z_2$. Thus the pair $(X_{f,+}, X_{f, -})$ has arbitrarily large  finite independence sets. This proves the ``if'' part.
\end{proof}

\begin{lemma} \label{L-null}
The pair $(X_{f,+}, X_{f, -})$ has arbitrarily large finite independence sets if and only if the pair $(X_+, X_-)$ has arbitrarily large finite independence sets.
\end{lemma}
\begin{proof} Since $\pi_f^{-1}(X_+)\subseteq X_{f, +}$ and $\pi_f^{-1}(X_-)\subseteq X_{f, -}$, the ``if'' part is trivial.

Assume that $(X_{f,+}, X_{f, -})$ has an independence set $M$ with cardinality $2m$ for some positive integer $m$. Then for any map $\omega: M\rightarrow \{+, -\}$ one has $\bigcap_{s\in M}s^{-1}X_{f, \omega(s)}\neq \emptyset$. For each such $\omega$, fix a point $z_\omega\in \bigcap_{s\in M}s^{-1}X_{f, \omega(s)}$. Then for each such $\omega$ and $s\in M$, one has $sz_\omega\in X_{f, \omega(s)}\subseteq \pi_f^{-1}(X_{\omega(s)})\cup \pi_f^{-1}(x_1)$. List the elements of $M$ as $s_1, s_2, \dots, s_{2m}$. Denote by $A$ the set of integers $1\le n\le m$ such that $s_{2n-1}z_\omega\in \pi_f^{-1}(x_1)$ for some $\omega$.

Let $n\in A$ and take one $\omega_n$ such that $s_{2n-1}z_{\omega_n}\in \pi_f^{-1}(x_1)$. Then $s_{2n-1}\pi_f(z_{\omega_n})=x_1$. Since the stabilizer group of $x_1$ is trivial, we have $$\pi_f(s_{2n}z_{\omega_n})=s_{2n}\pi_f(z_{\omega_n})=s_{2n}s_{2n-1}^{-1}x_1\neq x_1.$$
If $s_{2n}s_{2n-1}^{-1}x_1\in X_+$, then $s_{2n}z_{\omega_n}\in X_{f, +}$ and hence $\omega_n(s_{2n})=+$. Similarly, if $s_{2n}s_{2n-1}^{-1}x_1\in X_-$, then $s_{2n}z_{\omega_n}\in X_{f, -}$ and hence $\omega_n(s_{2n})=-$. Thus for any map $\omega:M\rightarrow \{+, -\}$, if $\omega(s_{2n})\neq \omega_n(s_{2n})$, then $s_{2n-1}z_\omega\in \pi_f^{-1}(X_{\omega(s_{2n-1})})$ and hence $s_{2n-1}\pi_f(z_\omega)\in X_{\omega(s_{2n-1})}$.

Set $M'=\{s_{2n-1}: n=1, \dots, m\}$. For each map $\omega': M'\rightarrow \{+, -\}$, extend it to a map $\omega: M\rightarrow \{+, -\}$ such that $\omega(s_{2n})\neq \omega_n(s_{2n})$ for all $n\in A$. By the above we have $s_{2n-1}\pi_f(z_\omega)\in X_{\omega(s_{2n-1})}=X_{\omega'(s_{2n-1})}$ for all $n=1, \dots, m$. Therefore $\bigcap_{s\in M'}s^{-1}X_{\omega'(s)}\neq \emptyset$, i.e. $M'$ is an independence set for $(X_+, X_-)$. This proves the ``only if'' part.
\end{proof}

\begin{remark} \label{R-tame}
The analogues of Proposition~\ref{P-construction null} and Lemmas~\ref{L-null1} and \ref{L-null} for tameness all hold with similar proofs.
\end{remark}

We come back to the construction of McMahon. Denote by $\sM^*(X)$ the set of {\it nonatomic} $\mu$ in $\sM(X)$, i.e. $\mu(\{x\})=0$ for every $x\in X$. Also denote by $\pi_{f*}$ the continuous surjective map $\sM(X_f)\rightarrow \sM(X)$ induced by $\pi_f$.

\begin{lemma} \label{L-nonatomic}
For each $\mu\in \sM^*(X)$, the set $(\pi_{f*})^{-1}(\mu)$ consists of a single point, which we denote by $\mu_f$. The map $\sM^*(X)\rightarrow \sM(X_f)$ sending $\mu$ to $\mu_f$ is continuous.
\end{lemma}
\begin{proof} Let $\nu\in \sM(X_f)$ with $\pi_{f*}(\nu)=\mu$. Since $\mu$ is nonatomic, we have
$$\nu(\pi_f^{-1}(\Gamma x_1))=\mu(\Gamma x_1)=0.$$

Restrict $\pi_f$ to $\pi_f^{-1}(X\setminus \Gamma x_1)$ we get a map $\psi:\pi_f^{-1}(X\setminus \Gamma x_1)\rightarrow X\setminus \Gamma x_1$.
We claim that $\psi$ is a  homeomorphism. Since $\psi$ is continuous and bijective, it suffices to show that $\psi^{-1}$ is continuous. Let $(x_j)_{j\in J}$ be a net in
$X\setminus \Gamma x_1$ converging to some $x_\infty\in X\setminus \Gamma x_1$. We just need to show that $\psi^{-1}(x_j)\to \psi^{-1}(x_\infty)$ as $j\to \infty$. Since $X_f$ is compact, passing to a subnet if necessary, we may assume that $\psi^{-1}(x_j)$ converges to some $z\in X_f$ as $j\to \infty$. Then $x_j=\pi_f(\psi^{-1}(x_j))$ converges to $\pi_f(z)$, and hence  $\pi_f(z)=x_\infty$. Thus $z=\psi^{-1}(x_\infty)$, and we conclude that $\psi^{-1}(x_j)\to \psi^{-1}(x_\infty)$  as desired. This proves our claim.

Since $\psi$ is a homeomorphism, it is a Borel isomorphism. Then for any Borel set $A\subseteq X_f$, we have
$$\nu(A)=\nu(A\setminus \pi_f^{-1}(\Gamma x_1))=\mu(\psi(A\setminus \pi_f^{-1}(\Gamma x_1)))=\mu(\pi_f(A)\setminus \Gamma x_1).$$
Therefore $\nu$ is unique.

Denote by $\varphi$ the map $\sM^*(X)\rightarrow \sM(X_f)$ sending $\mu$ to $\mu_f$. We shall show that $\varphi$ is continuous. Let $\{\mu_j\}_{j\in J}$ be a net in $\sM^*(X)$ converging to some $\mu_\infty\in \sM^*(X)$. We just need to show that $\varphi(\mu_j)\to \varphi(\mu_\infty)$ as $j\to \infty$. Since $\sM(X_f)$ is compact, passing to a subnet if necessary, we may assume that $\varphi(\mu_j)$ converges to some $\nu\in \sM(X_f)$ as $j\to \infty$. Then $\mu_j=\pi_{f*}(\varphi(\mu_j))$ converges to $\pi_{f*}(\nu)$, and hence
$\pi_{f*}(\nu)=\mu_\infty$. Thus $\nu=\varphi(\mu_\infty)$, and we conclude that $\varphi(\mu_j)\to \varphi(\mu_\infty)$ as desired. Therefore $\varphi$ is continuous.
\end{proof}

Now assume that $\Gamma\curvearrowright Y$  and $\Gamma\curvearrowright Z$ are continuous actions on compact metrizable spaces
such that $X=Y\times Z$ and the action $\Gamma\curvearrowright X$ is the product action $\Gamma\curvearrowright Y\times Z$. Also assume that there is some $\Gamma$-invariant nonatomic $\mu_Z\in \sM(Z)$. Denote by $\pi_Y$ the projection $X\rightarrow Y$.
For each $y\in Y$, we have the nonatomic measure $\delta_y\times \mu_Z\in \sM(X)$, where $\delta_y$ denotes the point mass at $y$.
Then $y\mapsto \delta_y\times \mu_Z$ is a RIM for the extension $\pi_Y$.
Thus by Lemma~\ref{L-nonatomic} the map $y\mapsto (\delta_y\times \mu_Z)_f$ is a RIM for the extension $\pi_Y\circ \pi_f$.

In \cite{McMahon76, McMahon78} McMahon took $\Gamma=G\times \Zb$ for $G$ being any dense countable subgroup of the $p$-adic integer group $\Zb_p$, and $Y=Z=\Zb_p$. His actions $\Gamma\curvearrowright Y$ and $\Gamma\curvearrowright Z$ are the ones factoring through the shift actions $G\curvearrowright Y$ and $\Zb\curvearrowright Z$ via treating $\Zb$ as a dense subgroup of $\Zb_p$ naturally. His measure $\mu_Z$ is the normalized Haar measure of $Z$. Taking suitable choices of $f$, he showed that the extension $\pi_Y\circ \pi_f$ could be either open or non-open.

\section{Residually finite groups} \label{S-RF}

As a warm up, in this section we apply the construction in Section~\ref{S-construction} to residually finite groups, and show that sometimes it yields null actions. Though the results of this section will not be used for the proof of Theorem~\ref{T-main}, the method in this section will be used in Section~\ref{S-free} in a much more complicated way.

Let $\Gamma$ be a countably infinite residually finite group with identity element $e_\Gamma$. This means that there is  a strictly decreasing sequence $\{\Gamma_n\}$ of finite-index normal subgroups of $\Gamma$ such that $\bigcap_{n\in \Nb} \Gamma_n=\{e_\Gamma\}$.

Denote by $X$ the inverse limit $\varprojlim_{n\to \infty}\Gamma/\Gamma_n$, which is the subset of $\prod_{n\in \Nb}\Gamma/\Gamma_n$ consisting of $(x_n)_{n\in \Nb}$ satisfying $\pi_{n, n+1}(x_{n+1})=x_n$ for all $n\in \Nb$. Here $\pi_{n, n+1}$ denotes the natural homomorphism $\Gamma/\Gamma_{n+1}\rightarrow \Gamma/\Gamma_n$.
This is a  compact metrizable totally disconnected group. Denote by $\pi_n$ the quotient map $X\rightarrow \Gamma/\Gamma_n$. The group $\Gamma$ is naturally a subgroup of $X$, and hence has a natural left shift action on $X$. Clearly the action $\Gamma\curvearrowright X$ is minimal and free in the sense that every point of $X$ has trivial stabilizer group. Note that $X$ is a Cantor set, so it contains more than one orbit.

\begin{lemma} \label{L-equicontinuous null}
The action $\Gamma\curvearrowright X$ is null.
\end{lemma}
\begin{proof} The compact metrizable group $X$ has a translation-invariant metric, which is invariant under the $\Gamma$-action. It follows from (2) of Proposition~\ref{P-null} that $\Gamma\curvearrowright X$ is null.
\end{proof}

For each $n\ge 2$ take $\gamma_n\in \Gamma_{n-1}\setminus \Gamma_n$.  Set $C_n=\pi_n^{-1}(\gamma_n\Gamma_n)$, which is a closed and open subset of $X$. The sets $C_n$ for $n\ge 2$ are pairwise disjoint, $e_\Gamma\not\in C_n$, and $C_n\to \{e_\Gamma\}$ as $n\to \infty$ in the sense that for every neighborhood $U$ of $e_\Gamma$ in $X$ one has $C_n\subseteq U$ for all large enough $n$.
Set $X_+=\bigcup_{n\ge 2}C_n$ and $X_-=X\setminus (X_+\cup\{e_\Gamma\})$.

\begin{lemma} \label{L-RF3}
Each independence set $M\subseteq \Gamma$ for $(X_+, X_-)$ has cardinality at most $5$.
\end{lemma}

 We leave the proof of Lemma~\ref{L-RF3} to the end of this section. Consider the function $f: X\setminus \{e_\Gamma\}\rightarrow \{1, -1\}$ defined by $f(x)=1$ if $x\in X_+$ and $f(x)=-1$ if $x\in X_-$. Since $C_n\to \{e_\Gamma\}$ as $n\to \infty$, the function $f$ is continuous. Assume further that $[\Gamma_n:\Gamma_{n+1}]>2$ for all $n$. Then every neighborhood of $e_\Gamma$ in $X$ intersects with both $X_+$ and $X_-$. Thus $f$ cannot be extended to $X$ continuously.
 Then we can apply McMahon's construction in Section~\ref{S-construction} to obtain the minimal action $\Gamma\curvearrowright X_f$ and the fact map $\pi_f: X_f\rightarrow X$.
 From Proposition~\ref{P-construction null} we conclude

\begin{theorem} \label{T-RF}
The action $\Gamma\curvearrowright X_f$ is null.
\end{theorem}

\begin{remark} \label{R-integer}
Fix distinct prime numbers $p, q\ge 3$. For $\Gamma=\Zb$, we can take $\Gamma_n=p^nq^n\Zb$. Then $X$ is the product of the $p$-adic integer group $\Zb_p$ and the $q$-adic integer group $\Zb_q$. Denote by $\pi_p$ the projection $X\rightarrow \Zb_p$. As the normalized Haar measure of $\Zb_q$ is nonatomic, from the discussion at the end of Section~\ref{S-construction} we know that the extension $\pi_p\circ \pi_f$ has a RIM. The proofs of Lemmas~\ref{L-RIM} and \ref{L-open} in the next section also work in this case to show that $\pi_p\circ \pi_f$ has a unique RIM and is not open.
\end{remark}

To prove Lemma~\ref{L-RF3}, we need to make some preparation.

\begin{lemma} \label{L-RF1}
Let $s_1, s_2\in \Gamma$ and $x, y\in X$ such that $s_1x\in C_{n_1}$ and $s_2x\in C_{n_2}$ with $n_1<n_2$, and $s_1y\not \in C_{n_1}$ and $s_2y\in C_{m_2}$. Then $m_2\le n_1$.
\end{lemma}
\begin{proof} Set $t=s_1s_2^{-1}$. Then $tC_{n_2}\cap C_{n_1}\neq \emptyset$, and  hence
$$t\Gamma_{n_1}=t\gamma_{n_2}\Gamma_{n_1}=\gamma_{n_1}\Gamma_{n_1}.$$
That is, $t\in \gamma_{n_1}\Gamma_{n_1}$. If $m_2>n_1$, then
$$\pi_{n_1}(s_1y)=\pi_{n_1}(ts_2y)=t\pi_{n_1}(s_2y)=t\Gamma_{n_1}=\gamma_{n_1}\Gamma_{n_1},$$
and hence $s_1y\in C_{n_1}$, which is a contradiction. Therefore $m_2\le n_1$.
\end{proof}

\begin{lemma} \label{L-RF2}
Let $s_1, s_2, s_3\in \Gamma$ and $x\in X$ such that $s_1x\in C_{n_1}, s_2x\in C_{n_2}$ and $s_3x\in C_{n_3}$ with $n_1<\min(n_2, n_3)$. Then there is no $y\in X$ satisfying $s_1y, s_2y\not\in X_+$ and $s_3y\in X_+$.
\end{lemma}
\begin{proof} For each $i=2, 3$  we have $s_i\pi_{n_i}(x)=\gamma_{n_i}\Gamma_{n_i}$, and hence $s_i\pi_{n_1}(x)=\Gamma_{n_1}$. It follows that $s_2\Gamma_{n_1}=s_3\Gamma_{n_1}$.

Suppose that for some $y\in X$ we have $s_1y, s_2y\not\in X_+$ and $s_3y\in X_+$. Say, $s_3y\in C_{m_3}$.
Applying Lemma~\ref{L-RF1} to $s_1, s_3\in \Gamma$ and $x, y\in X$ we have $m_3\le n_1$. Thus $s_2\Gamma_{m_3}=s_3\Gamma_{m_3}$, and hence $s_2\pi_{m_3}(y)=s_3\pi_{m_3}(y)$. Since $s_2y\not\in C_{m_3}$ and $s_3y\in C_{m_3}$, we have $s_2\pi_{m_3}(y)\neq \gamma_{m_3}\Gamma_{m_3}$ and $s_3\pi_{m_3}(y)=\gamma_{m_3}\Gamma_{m_3}$, which is a contradiction.
\end{proof}

We are ready to prove Lemma~\ref{L-RF3}.

\begin{proof}[Proof of Lemma~\ref{L-RF3}] Assume that $(X_+, X_-)$ has an independence set $M$ with cardinality $6$. For each map $\omega: M\rightarrow \{+, -\}$, fix
a point $x_\omega\in \bigcap_{s\in M}s^{-1}X_{\omega(s)}$. For any such $\omega$ and any $s\in \omega^{-1}(+)$, the point $sx_{\omega}$ lies in $C_n$ for a unique $n\ge 2$. Denote this $n$ by $g(\omega, s)$. By Lemma~\ref{L-RF2} for each such $\omega$ there is a set $B_\omega\subseteq \omega^{-1}(+)$ with cardinality at most $1$ such that the function $s\mapsto g(\omega, s)$ is constant on $\omega^{-1}(+)\setminus B_\omega$.

Take distinct $s_1, s_2\in M$. Define a map $\omega: M\rightarrow \{+, -\}$ by $\omega(s_1)=\omega(s_2)=-$ and $\omega(s)=+$ for all $s\in M\setminus \{s_1, s_2\}$. Then $\omega^{-1}(+)\setminus B_\omega$ has cardinality at least $3$. Say, $g(\omega, s)=n_1$ for all $s\in \omega^{-1}(+)\setminus B_\omega$.

Take $s_3\in \omega^{-1}(+)\setminus B_\omega$. Define a map $\tilde{\omega}: M\rightarrow \{+, -\}$ by $\tilde{\omega}(s_3)=-$ and $\tilde{\omega}(s)=+$ for all $s\in M\setminus \{s_3\}$. Say, $g(\tilde{\omega}, s)=n_2$ for all $s\in \tilde{\omega}^{-1}(+)\setminus B_{\tilde{\omega}}$.
Then $\tilde{\omega}^{-1}(+)\setminus B_{\tilde{\omega}}$ has nonempty intersection with both $\{s_1, s_2\}$ and $\omega^{-1}(+)\setminus (B_\omega \cup \{s_3\})$. Without loss of generality, we may assume $s_1\in \tilde{\omega}^{-1}(+)\setminus B_{\tilde{\omega}}$. Take $s_4\in (\tilde{\omega}^{-1}(+)\setminus B_{\tilde{\omega}})\cap (\omega^{-1}(+)\setminus (B_\omega \cup \{s_3\}))$.

Now we have $s_1, s_3, s_4$ pairwise distinct. We also have $s_3x_\omega, s_4x_\omega\in C_{n_1}$, $s_1x_{\tilde{\omega}}, s_4x_{\tilde{\omega}}\in C_{n_2}$, and $s_1x_\omega, s_3x_{\tilde{\omega}}\not\in X_+$.

From $s_3x_\omega, s_4x_\omega\in C_{n_1}$ and $s_1x_\omega\not\in C_{n_1}$ we have
$$s_3\pi_{n_1}(x_\omega)=\gamma_{n_1}\Gamma_{n_1}=s_4\pi_{n_1}(x_\omega)\neq s_1\pi_{n_1}(x_\omega),$$
and hence $s_3\Gamma_{n_1}=s_4\Gamma_{n_1}\neq s_1\Gamma_{n_1}$. Similarly, from $s_1x_{\tilde{\omega}}, s_4x_{\tilde{\omega}}\in C_{n_2}$ and $s_3x_{\tilde{\omega}}\not\in C_{n_2}$ we have $s_1\Gamma_{n_2}=s_4\Gamma_{n_2}\neq s_3\Gamma_{n_2}$.

If $n_1\ge n_2$, then from $s_3\Gamma_{n_1}=s_4\Gamma_{n_1}$ we have $s_3\Gamma_{n_2}=s_4\Gamma_{n_2}$, which is a contradiction.
If $n_2\ge n_1$, then from $s_1\Gamma_{n_2}=s_4\Gamma_{n_2}$ we have $s_1\Gamma_{n_1}=s_4\Gamma_{n_1}$, which is also a contradiction.
\end{proof}

\section{Free groups} \label{S-free}

In this section we prove Theorem~\ref{T-main}.

Let $r\ge 2$ and $\Gamma=F_r$ be the free group with $r$ generators $S=\{a, b, a_3, \dots, a_r\}$.

Denote by $Y$ the Gromov boundary of $\Gamma$.
This is the set of all infinite reduced words in $S\cup S^{-1}$, i.e. the set of elements $x=(x_n)_{n\in \Nb}$ in $(S\cup S^{-1})^\Nb$ satisfying $x_{n+1}\neq x_n^{-1}$ for all $n\in \Nb$. It is a closed subset of $(S\cup S^{-1})^\Nb$, hence compact metrizable. The group $\Gamma$ acts on $Y$ continuously by concatenation and cancellation.
Clearly this action is minimal and effective. By \cite[page 161]{Glasner75a} this action  is also strongly proximal.

It is well known that $\Gamma$ is residually finite \cite[Corollary C-1.126]{Rotman}. Let $\{\Gamma_n\}$ be a strictly decreasing sequence of finite-index normal subgroups of $\Gamma$ with $\bigcap_{n\in \Nb}\Gamma_n=\{e_\Gamma\}$. Set $Z=\varprojlim_{n\to \infty}\Gamma/\Gamma_n$ and $X=Y\times Z$. As in Section~\ref{S-RF}, $\Gamma$ is naturally a subgroup of the compact metrizable totally disconnected group $Z$ and has a natural left shift action on $Z$. This action $\Gamma\curvearrowright Z$ is minimal and free.
Then the product action $\Gamma\curvearrowright X$ is also free. Denote by $\mu_Z$ the normalized Haar measure on $Z$, which is nonatomic and $\Gamma$-invariant.
For each $n\in \Nb$, denote by $\pi_n$ the natural homomorphism $Z\rightarrow \Gamma/\Gamma_n$.

\begin{lemma} \label{L-minimal}
The product action $\Gamma\curvearrowright X$ is minimal.
\end{lemma}
\begin{proof} An action $\Gamma\curvearrowright Z'$ is called {\it weakly non-contractible} if it is minimal and there is some $\mu\in \sM(Z')$ with support $Z'$ such that the orbit closure of $\mu$ in $\sM(Z')$ is minimal. Every weakly non-contractible action $\Gamma\curvearrowright Z'$ is disjoint from every minimal strongly proximal action $\Gamma\curvearrowright Y'$ in the sense that the product action $\Gamma\curvearrowright Y'\times Z'$ is minimal \cite[Theorem 6.1]{Glasner75a}. Since $\mu_Z$ has support $Z$ and is $\Gamma$-invariant, we know that $\Gamma\curvearrowright Z$ is weakly non-contractible. Therefore the product action $\Gamma\curvearrowright Y\times Z$ is minimal.
\end{proof}

For each nontrivial $s\in \Gamma$, we say that $y\in Y$ {\it starts with $s$} if $y=sy'$ for some $y'\in Y$ such that the last letter of (the reduced form of) $s$ is different from the inverse of the first letter of $y'$. Similarly, we shall talk about $t\in \Gamma$ starting or ending with $s$.  For any nontrivial $s\in \Gamma$, denote by $V_s$ the set of elements in $Y$ starting with $s$. Denote by $a^\infty$ the element in $Y$ taking constant value $a$.

For each $n\ge 2$ take $\gamma_n\in \Gamma_{n-1}\setminus \Gamma_n$, and set $u_n=a^nba^{-n}b^{-1}$, $D_n=V_{u_n}$ and $C_n=\pi_n^{-1}(\gamma_n\Gamma_n)$. Then $D_n\times C_n$ is a closed and open subset of $X$. The sets $D_n\times C_n$ are pairwise disjoint, $(a^\infty, e_\Gamma)\not\in D_n\times C_n$, and $D_n\times C_n\rightarrow \{(a^{\infty}, e_\Gamma)\}$ as $n\to \infty$ in the sense that for every neighborhood $U$ of $(a^\infty, e_\Gamma)$ in $X$ one has $D_n\times C_n\subseteq U$ for all large enough $n\in \Nb$. Set $X_+=\bigcup_{n\ge 2}(D_n\times C_n)$ and $X_-=X\setminus (X_+\cup\{(a^\infty, e_\Gamma)\})$. Then every neighborhood of $(a^\infty, e_\Gamma)$ in $X$ intersects with both $X_+$ and $X_-$.

 Define a function $f: X\setminus \{(a^\infty, e_\Gamma)\}\rightarrow \{1, -1\}$ by $f(x)=1$ if $x\in X_+$ and $f(x)=-1$ if $x\in X_-$. Since $D_n\times C_n\to \{(a^\infty, e_\Gamma)\}$ as $n\to \infty$, the function $f$ is continuous.
 Since every neighborhood of $(a^\infty, e_\Gamma)$ in $X$ intersects with both $X_+$ and $X_-$, the function $f$ cannot be extended to $X$ continuously.
Therefore we can apply McMahon's construction in Section~\ref{S-construction} to obtain the minimal action $\Gamma\curvearrowright X_f$ and the factor map $\pi_f: X_f\rightarrow X$.

Denote by $\pi_Y$ the projection $X\rightarrow Y$.

\begin{lemma} \label{L-RIM}
The extension $\pi_Y\circ \pi_f$ has a unique RIM.
\end{lemma}
\begin{proof} From our discussion at the end of Section~\ref{S-construction} and using the notation there, we know that $y\mapsto \delta_y\times \mu_Z$ is a RIM for $\pi_Y$ and
$y\mapsto (\delta_y\times \mu_Z)_f$ is a RIM for  $\pi_Y\circ \pi_f$.

An extension $\pi': X'\rightarrow Y'$ between continuous actions $\Gamma\curvearrowright X'$ and $\Gamma\curvearrowright Y'$ is called a {\it group extension} if there is a  compact metrizable group $Z'$ acting continuously on $X'$ denoted by $(x', z')\rightarrow x'z'$ for $x'\in X'$ and $z'\in Z'$ such that $s(x'z')=(sx')z'$ for all $s\in \Gamma, x'\in X', z'\in Z'$, and $(\pi')^{-1}(\pi'(x'))=x'Z'$ for all $x'\in X'$. Every group extension between minimal actions has a unique RIM \cite[Corollary 3.7]{Glasner75b}.

Clearly $\pi_Y$ is a group extension. By Lemma~\ref{L-minimal} the action $\Gamma\curvearrowright X$ is minimal. Therefore $\pi_Y$ has a unique RIM, which must be $y\mapsto \delta_y\times \mu_Z$.

Let $y\mapsto \mu_y$ be a RIM for $\pi_Y\circ \pi_f$. Then $y\mapsto \pi_{f*}(\mu_y)$ is a RIM for $\pi_Y$. Thus $\pi_{f*}(\mu_y)=\delta_y\times \mu_Z$ for every $y\in Y$. By Lemma~\ref{L-nonatomic} we get $\mu_y=(\delta_y\times \mu_Z)_f$ for every $y\in Y$. Therefore $\pi_Y\circ \pi_f$ has a unique RIM.
\end{proof}

\begin{lemma} \label{L-distal}
The extension $\pi_Y\circ \pi_f$ is point-distal.
\end{lemma}
\begin{proof} Denote by $\pi_Z$ the projection $X\rightarrow Z$.

Let $\tilde{x}\in \pi_f^{-1}(X\setminus \Gamma(a^\infty, e_\Gamma))$ and $\tilde{x}'\in X_f$ with $\pi_Y\circ \pi_f(\tilde{x})=\pi_Y\circ \pi_f(\tilde{x}')$ such that
the orbit closure of $(\tilde{x}, \tilde{x}')$ in $X_f^2$ intersects with the diagonal. Then the orbit closure of $(\pi_Z\circ \pi_f(\tilde{x}), \pi_Z\circ \pi_f(\tilde{x}'))$ in $Z^2$ intersects with the diagonal. Note that the compact metrizable group $Z$ has a translation-invariant metric, which is then invariant under the $\Gamma$-action. It follows that $ \pi_Z\circ \pi_f(\tilde{x})=\pi_Z\circ \pi_f(\tilde{x}')$. Therefore $\pi_f(\tilde{x})=\pi_f(\tilde{x'})$, and hence $\tilde{x}=\tilde{x}'$. Since $\Gamma\curvearrowright X$ is minimal by Lemma~\ref{L-minimal}, this means that every point in $\pi_f^{-1}(X\setminus \Gamma(a^\infty, e_\Gamma))$ witnesses the definition of point-distality.
\end{proof}

We remark that Glasner showed that every RIM extension between tame minimal actions is point-distal \cite[Theorem 4.4]{Glasner18}. Thus Lemma~\ref{L-distal} also follows directly once we show later that $\Gamma\curvearrowright X_f$ is null.

\begin{lemma} \label{L-open}
The map $\pi_Y\circ \pi_f$ is not open.
\end{lemma}
\begin{proof} Denote by $\tilde{f}$ the continuous  extension of $f\circ \pi_f: \pi_f^{-1}(X\setminus \{(a^\infty, e_\Gamma)\})\rightarrow \{1, -1\}$ to $X_f$. Set $X_{f, +}=\tilde{f}^{-1}(1)$.
Then $X_{f, +}$ is an open subset of $X_f$.
But
$$\pi_Y\circ \pi_f(X_{f,+})=\pi_Y(\{(a^\infty, e_\Gamma)\}\cup X_+)=\{a^\infty\}\cup \bigcup_{n\ge 2}D_n$$
is not open.  Therefore $\pi_Y\circ \pi_f$ is not open.
\end{proof}

We are left to show  that $\Gamma\curvearrowright X_f$ is null, which is also the most technical part. For this purpose we need to make some preparation.

\begin{lemma} \label{L-free6}
Let $t\in \Gamma$ and $y\in Y$. Assume that $t$ does not start with $u_n$.
Then $ty\in D_n$ if and only if $y$ starts with $t^{-1}u_n$.
\end{lemma}
\begin{proof} Since $t$ does not start with $u_n$, the element $t^{-1}u_n$ ends with $b^{-1}$.

Suppose that $y$ starts with $t^{-1}u_n$. Then $y=t^{-1}u_ny'$ for some $y'\in Y$ not starting with $b$. We have $ty=u_ny'$. Since $y'$ does not start with $b$, the point $u_ny'$ starts with $u_n$.  Thus $ty\in D_n$. This proves the ``if'' part.

Now suppose that $ty\in D_n$. Then $ty=u_ny'$ for some $y'\in Y$ not starting with $b$. We have $y=t^{-1}u_ny'$. Since $t^{-1}u_n$ ends with $b^{-1}$ and $y'$ does not start with $b$, $t^{-1}u_ny'$ starts with $t^{-1}u_n$. This proves the ``only if'' part.
\end{proof}

\begin{lemma} \label{L-free7}
Let $t\in \Gamma$ and $y\in Y$. Assume that $t$ starts with $u_n$.
Then $ty\in D_n$ if and only if $y$ does not start with $t^{-1}u_nb$.
\end{lemma}
\begin{proof} We have $t=u_ns$ for some $s\in \Gamma$ not starting with $b$. Then $ty\not\in D_n$ if and only if $y$ starts with $(b^{-1}s)^{-1}=t^{-1}u_nb$.
\end{proof}

\begin{lemma} \label{L-free2}
Let $t\in \Gamma$ be nontrivial such that $tD_{n_2}\cap D_{n_1}\neq \emptyset$. Then one of the following $3$ situations must hold:
\begin{enumerate}
\item $t$ ends with $u_{n_2}^{-1}=ba^{n_2}b^{-1}a^{-n_2}$,
\item $t$ starts with $u_{n_1}=a^{n_1}ba^{-n_1}b^{-1}$,
\item $n_2\neq n_1$ and $t=u_{n_1}u_{n_2}^{-1}=a^{n_1}ba^{n_2-n_1}b^{-1}a^{-n_2}$.
\end{enumerate}
\end{lemma}
\begin{proof} Assume that (1) and (2) do not hold. Take $y\in D_{n_2}$ with $ty\in D_{n_1}$. Then $y$ starts with $u_{n_2}$.
Since $t$ does not start with $u_{n_1}$, we know that $t^{-1}u_{n_1}$ ends with $b^{-1}$ and by Lemma~\ref{L-free6} the element $y$ also starts with $t^{-1}u_{n_1}$.
Then either $u_{n_2}$ starts with $t^{-1}u_{n_1}$ or $t^{-1}u_{n_1}$ starts with $u_{n_2}$. Since $t^{-1}u_{n_1}$ ends with $b^{-1}$, if  $u_{n_2}$ starts with $t^{-1}u_{n_1}$, then we must have $u_{n_2}=t^{-1}u_{n_1}$, i.e. (3) holds. If $t^{-1}u_{n_1}$ starts with $u_{n_2}$, since $t^{-1}$ does not start with $u_{n_2}$, then it is easy to see that we still have $t^{-1}u_{n_1}=u_{n_2}$.
\end{proof}

\begin{lemma} \label{L-free5}
Let $s_1, s_2\in \Gamma$ and $z, z'\in Z$ such that $s_1z, s_2z\in C_{n}$, and  $s_2z'\in C_{m}$. Then $m=n$ if and only if $s_1z'\in C_n$.
\end{lemma}
\begin{proof} Set $t=s_1s_2^{-1}$. Then $tC_{n}\cap C_{n}\neq \emptyset$, and  hence
$t\gamma_{n}\Gamma_{n}=\gamma_{n}\Gamma_{n}$. It follows that
$t\in \Gamma_{n}$.

Assume that $m=n$. Then $s_1z'=t(s_2z')\in tC_n=C_n$. This proves the ``only if'' part.

Conversely, assume that $s_1z'\in C_n$. Then $s_2z'=t^{-1}(s_1z')\in t^{-1}C_n=C_n$, and hence $m=n$. This proves the ``if'' part.
\end{proof}

Note that for each $n\in \Nb$, $\Gamma_n$ is torsion-free. Since $\bigcap_{m\in \Nb}\Gamma_m=\{e_\Gamma\}$, for each $n\in \Nb$, when $m$ is large enough, $\Gamma_n/\Gamma_{n+m}$ has nontrivial elements with order strictly bigger than $2$. Thus replacing $\{\Gamma_n\}$ by a suitable subsequence, we may choose $\gamma_n\in \Gamma_{n-1}\setminus \Gamma_n$ such that $\gamma_n^2\not\in \Gamma_n$ for every $n\ge 2$, which we shall assume from now on.

\begin{lemma} \label{L-free1}
Let $s_1, s_2\in \Gamma$ and $z, z'\in Z$ such that $s_1z\in C_{n_1}$ and $s_2z\in C_{n_2}$ with $n_1<n_2$, and $s_1z'\in C_{m_1}$ and $s_2z'\in C_{m_2}$. Then $s_1s_2^{-1}\in \gamma_{n_1}\Gamma_{n_1}$, and  one of the following two situations must hold:
\begin{enumerate}
\item $m_1=n_1<m_2$,
\item $m_1=m_2<n_1$.
\end{enumerate}
\end{lemma}
\begin{proof} Set $t=s_1s_2^{-1}$. Then $tC_{n_2}\cap C_{n_1}\neq \emptyset$, and  hence
$$t\Gamma_{n_1}=t\gamma_{n_2}\Gamma_{n_1}=\gamma_{n_1}\Gamma_{n_1}.$$
That is, $t\in \gamma_{n_1}\Gamma_{n_1}$. We consider three cases.

Consider first the case $m_1<m_2$.  Similarly we have $t\in \gamma_{m_1}\Gamma_{m_1}$. Then $t\in C_{n_1}\cap C_{m_1}$, and hence $m_1=n_1$.

Next we consider the case $m_1=m_2$. Then $tC_{m_1}\cap C_{m_1}\neq \emptyset$, and hence $t\in \Gamma_{m_1}$, which implies $m_1<n_1$.

Finally we consider the case $m_1>m_2$.  Similarly we have $t^{-1}\in \gamma_{m_2}\Gamma_{m_2}$. Then $t\in \gamma_{n_1}\Gamma_{n_1}\cap \gamma_{m_2}^{-1}\Gamma_{m_2}$, and hence $n_1=m_2$. It follows that $\gamma_{n_1}\Gamma_{n_1}=\gamma_{n_1}^{-1}\Gamma_{n_1}$, and hence $\gamma_{n_1}^2\in \Gamma_{n_1}$, which is a contradiction to our assumption. Therefore the case $m_1>m_2$ does not happen.
\end{proof}

\begin{lemma} \label{L-free3}
Let $s_1, s_2\in \Gamma$ be distinct such that $s_1^{-1}X_+\cap s_2^{-1}X_+\neq \emptyset$.
 Set $t=s_1s_2^{-1}$. Then at least one of the following 10 situations must hold:
\begin{enumerate}
\item Type (A1) for $(s_1, s_2)$: $t=u_mvu_{n_2}^{-1}$ for some $2\le m<n_1<n_2$ determined by $t$ and  some nontrivial $v\in \Gamma$ not starting with $b$ and not ending with $b^{-1}$; for any $x''=(y'', z'')\in X$ with $s_2x''\in D_{k_2}\times C_{k_2}$, we have $s_1x''\in X_+$ if and only if either $k_2=m$, in which case $s_1x''\in D_{m}\times C_{m}$, or $k_2=n_2$ and
    $s_2y''$ starts with $t^{-1}u_{n_1}$, in which case $s_1x''\in D_{n_1}\times C_{n_1}$.
\item Type (A2) for $(s_1, s_2)$: $t=u_{n_1}vu_m^{-1}$ for some $2\le m<n_1$ determined by $t$ and some nontrivial $v\in \Gamma$ not starting with $b$ and not ending with $b^{-1}$; for any $x''=(y'', z'')\in X$ with $s_2x''\in D_{k_2}\times C_{k_2}$, we have $s_1x''\in X_+$ if and only if either $k_2>n_1$, in which case $s_1x''\in D_{n_1}\times C_{n_1}$, or $k_2=m$ and $s_2y''$ starts with $t^{-1}u_m$, in which case $s_1x''\in D_{m}\times C_{m}$.
\item  Type (B1) for $(s_1, s_2)$: $t=u_{n_1}vu_m^{-1}$ for some $2\le n_1<m$ determined by $t$ and some nontrivial $v\in \Gamma$ not starting with $b$ and not ending with $b^{-1}$; for any $x''=(y'', z'')\in X$ with $s_2x''\in D_{k_2}\times C_{k_2}$, we have $s_1x''\in X_+$ if and only if $s_1x''\in D_{n_1}\times C_{n_1}$ if and only if either $m\neq k_2>n_1$,  or $k_2=m$ and $s_2y''$ does not start with $t^{-1}u_{n_1}b$.
\item Type (B2) for $(s_1, s_2)$: $t=u_{n_1}v$  for some $n_1\ge 2$ determined by $t$ and some $v\in \Gamma$ not starting with $b$ and not ending with $u_m^{-1}$ for any $m>n_1$; for any $x''=(y'', z'')\in X$ with $s_2x''\in D_{k_2}\times C_{k_2}$, we have $s_1x''\in X_+$ if and only if  $k_2>n_1$ and $t\neq u_{n_1}a^{-k_2}$, in which case $s_1x''\in D_{n_1}\times C_{n_1}$.
\item Type (B3) for $(s_1, s_2)$: $t=u_{n_1}u_{n_2}^{-1}$ for some $2\le n_1<n_2$ determined by $t$; for any $x''=(y'', z'')\in X$ with $s_2x''\in D_{k_2}\times C_{k_2}$, we have $s_1x''\in X_+$ if and only if $k_2=n_2$, in which case $s_1x''\in D_{n_1}\times C_{n_1}$.
\item Type (B4) for $(s_1, s_2)$: $t=vu_{n_2}^{-1}$ for some $2\le n_1<n_2$ determined by $t$ and some $v\in \Gamma$ not ending with $b^{-1}$ and not starting with $u_{n_1}$; for any $x''=(y'', z'')\in X$ with $s_2x''\in D_{k_2}\times C_{k_2}$, we have $s_1x''\in X_+$ if and only if $k_2=n_2$ and $s_2y''$ starts with $t^{-1}u_{n_1}$, in which case $s_1x''\in D_{n_1}\times C_{n_1}$.
\item Type (C1) for $(s_1, s_2)$: $t=u_{n_1}vu_{n_2}^{-1}$ for some distinct $n_1, n_2\ge 2$ determined by $t$ and  some nontrivial $v\in \Gamma$ not starting with $b$ and not ending with $b^{-1}$; for any $x''=(y'', z'')\in X$ with $s_2x''\in D_{k_2}\times C_{k_2}$, we have $s_1x''\in X_+$ if and only if either $k_2=n_1$, in which case $s_1x''\in D_{n_1}\times C_{n_1}$, or $k_2=n_2$ and
    $s_2y''$ starts with $t^{-1}u_{n_2}$, in which case $s_1x''\in D_{n_2}\times C_{n_2}$.
\item  Type (C2) for $(s_1, s_2)$: $t=u_mvu_m^{-1}$ for some $m\ge 2$ determined by $t$ and some nontrivial $v\in \Gamma$ not starting with $b$ and not ending with $b^{-1}$; for any $x''=(y'', z'')\in X$ with $s_2x''\in D_{k_2}\times C_{k_2}$, we have $s_1x''\in X_+$ if and only if $k_2=m$ and $s_2y''$ does not start with $t^{-1}u_mb$, in which case $s_1x''\in D_m\times C_m$.
\item Type (C3) for $(s_1, s_2)$: $t=vu_m^{-1}$ for some $m\ge 2$ determined by $t$ and some $v\in \Gamma$ not starting with $u_m$ and not ending with $b^{-1}$; for any $x''=(y'', z'')\in X$ with $s_2x''\in D_{k_2}\times C_{k_2}$, we have  $s_1x''\in X_+$ if and only if $k_2=m$ and $s_2y''$ starts with $t^{-1}u_m$, in which case $s_1x''\in D_m\times C_m$.
\item The pair $(s_2, s_1)$ has one of the above types.
\end{enumerate}
\end{lemma}
\begin{proof} We consider first the case that there exist $x=(y, z), x'=(y', z')\in X$ such that $s_1x\in D_{n_1}\times C_{n_1}$, $s_2x\in D_{n_2}\times C_{n_2}$ and $s_1x', s_2x'\in D_m\times C_m$ with $n_1\neq n_2$. By symmetry we may assume that $n_1<n_2$. By Lemma~\ref{L-free1} we have $t\in \gamma_{n_1}\Gamma_{n_1}\subseteq \Gamma_{n_1-1}$ and $m<n_1$. Since $tD_m\cap D_m\neq \emptyset$, by Lemma~\ref{L-free2} either $t$ starts with $u_m$ or $t$ ends with $u_m^{-1}$.
We separate the case into two subcases.

Consider first the subcase $t$ ends with $u_{n_2}^{-1}$. Then $t$ must start with $u_m$.
Thus $t=u_mvu_{n_2}^{-1}$  for some nontrivial $v\in \Gamma$ not starting with $b$ and not ending with $b^{-1}$. We say that $(s_1, s_2)$ has type (A1). Clearly $m$ and $n_2$ are determined by $t$. Also $n_1$ is determined by $t$ as $t\in C_{n_1}$.
Let $x''=(y'', z'')\in X$ such that $s_2x''\in D_{k_2}\times C_{k_2}$.
If $k_2\neq n_2$, then $s_1y''=t(s_2y'')\in D_{m}$. Thus if $k_2\neq n_2$ and $s_1x''\in X_+$ then $s_1x''\in D_m\times C_m$ and hence $k_2=m$ by Lemma~\ref{L-free1}. Conversely, if $k_2=m$, then by Lemma~\ref{L-free5} we do have $s_1x''\in D_m\times C_m$.
If $k_2=n_2$ and $s_1x''\in X_+$, then by Lemma~\ref{L-free1} we have $s_1x''\in D_{n_1}\times C_{n_1}$ and hence by Lemma~\ref{L-free6} the point $s_2y''$ must start with $t^{-1}u_{n_1}=u_{n_2}v^{-1}u_m^{-1}u_{n_1}$. Conversely, if $k_2=n_2$ and $s_2y''$ starts with $t^{-1}u_{n_1}$, then by Lemma~\ref{L-free6} we do have $s_1x''\in D_{n_1}\times C_{n_1}$.

Next consider the subcase $t$ does not end  with $u_{n_2}^{-1}$. Since $tD_{n_2}\cap D_{n_1}\neq \emptyset$, by Lemma~\ref{L-free2} either $t=u_{n_1}u_{n_2}^{-1}$ or $t$ starts with $u_{n_1}$. Then $t$ cannot start with $u_m$, and hence $t$ ends with $u_m^{-1}$. It follows that $t$ starts with $u_{n_1}$.
Thus $t=u_{n_1}vu_m^{-1}$ for some
nontrivial $v\in \Gamma$ not starting with $b$ and not ending with $b^{-1}$. We say that $(s_1, s_2)$ has type (A2).   Clearly $m$ and $n_1$ are determined by $t$. Let $x''=(y'', z'')\in X$ such that $s_2x''\in D_{k_2}\times C_{k_2}$.
If $k_2\neq m$, then $s_1y''=t(s_2y'')\in D_{n_1}$. Thus if $k_2\neq m$ and $s_1x''\in X_+$, then $s_1x''\in D_{n_1}\times C_{n_1}$ and hence $k_2>n_1$ by Lemma~\ref{L-free1}. Conversely, if $k_2>n_1$, then we do have $s_1x''\in D_{n_1}\times C_{n_1}$. If $k_2=m$ and $s_1x''\in X_+$, then by Lemma~\ref{L-free1} we have $s_1x''\in D_{m}\times C_{m}$ and hence by Lemma~\ref{L-free6} the point $s_2y''$ starts with $t^{-1}u_m=u_mv^{-1}u_{n_1}^{-1}u_m$. Conversely, if $k_2=m$ and $s_2y''$ starts with $t^{-1}u_m$, then by Lemma~\ref{L-free6} we do have $s_1x''\in D_{m}\times C_{m}$.

Now consider the case there exist $x=(y, z)\in X$ such that $s_1x\in D_{n_1}\times C_{n_1}$, $s_2x\in D_{n_2}\times C_{n_2}$ with $n_1\neq n_2$, but
there is no $x'\in X$ such that $s_1x', s_2x'\in D_{m}\times C_{m}$. By symmetry we may assume that $n_1<n_2$.
By Lemma~\ref{L-free1} we have $t\in \gamma_{n_1}\Gamma_{n_1}$, and for any $x'=(y', z')\in X$ such that $s_1x'\in D_{m_1}\times C_{m_1}$ and $s_2x'\in D_{m_2}\times C_{m_2}$ we have
\begin{align} \label{E-B}
m_1=n_1<m_2.
\end{align}
 We separate the case into four subcases.

We consider first the subcase that there are  some $x^\sharp=(y^\sharp, z^\sharp), x'=(y', z')\in X$ such that $s_1x^\sharp, s_1x'\in D_{n_1}\times C_{n_1}$ and $s_2x^\sharp\in D_{m_2}\times C_{m_2}$ and $s_2x'\in D_{m}\times C_{m}$ with $n_1<\min(m_2, m)$ and $m_2\neq m$, and
$t$ ends with $u_m^{-1}$. Since $tD_{m_2}\cap D_{n_1}\neq \emptyset$, by Lemma~\ref{L-free2}
the element $t$ starts with $u_{n_1}$. Thus $t=u_{n_1}vu_m^{-1}$  for some nontrivial $v\in \Gamma$ not starting with $b$ and not ending with $b^{-1}$.
We say $(s_1, s_2)$ has type (B1). Clearly $n_1$ and $m$ are determined by $t$.
Let $x''=(y'', z'')\in X$ such that $s_2x''\in D_{k_2}\times C_{k_2}$.
If $s_1x''\in X_+$, then by \eqref{E-B} we have $k_2>n_1$ and $s_1x''\in D_{n_1}\times C_{n_1}$, and hence by Lemma~\ref{L-free7} the point $s_2y''$ does not start with $t^{-1}u_{n_1}b=u_mv^{-1}b$. Conversely, if $k_2>n_1$ and $s_2y''$ does not start with $t^{-1}u_{n_1}b$, then by Lemma~\ref{L-free7} we do have $s_1x''\in D_{n_1}\times C_{n_1}$.
In particular, if $m\neq k_2>n_1$, then $s_1x''\in D_{n_1}\times C_{n_1}$.

Next consider the subcase that for any $x'\in X$ such that $s_1x'\in D_{n_1}\times C_{n_1}$ and $s_2x'\in D_{m_2}\times C_{m_2}$ the element $t$ does not end with $u_{m_2}^{-1}$ and we do have some $x'=(y', z')\in X$ such that $s_1x'\in D_{n_1}\times C_{n_1}$ and $s_2x'\in D_{m_2}\times C_{m_2}$ with $m_2\neq n_2$.
Since $tD_{m_2}\cap D_{n_1}$ and $tD_{n_2}\cap D_{n_1}$ are nonempty, by Lemma~\ref{L-free2} the element $t$ must start with $u_{n_1}$. Say, $t=u_{n_1}v$  for some $v\in \Gamma$ not starting with $b$. We say $(s_1, s_2)$ has type (B2). Clearly $n_1$ is determined by $t$.
Note that $v$ cannot end with $u_k^{-1}$ for any $k>n_1$, otherwise we can easily find some $\tilde{x}\in X$ such that $s_1\tilde{x}\in D_{n_1}\times C_{n_1}$ and $s_2\tilde{x}\in D_{k}\times C_{k}$ contradicting to our assumption.
Let $x''=(y'', z'')\in X$ such that $s_2x''\in D_{k_2}\times C_{k_2}$. If $s_1x''\in X_+$, then by \eqref{E-B} we have $k_2>n_1$ and $s_1x''\in D_{n_1}\times C_{n_1}$, and hence by Lemma~\ref{L-free7} the point $s_2y''$ does not start with $t^{-1}u_{n_1}b=v^{-1}b$, which implies that $t\neq u_{n_1}a^{-k_2}$. Conversely, if $k_2>n_1$ and  $t\neq u_{n_1}a^{-k_2}$, then $s_2y''$ does not start with $t^{-1}u_{n_1}b=v^{-1}b$ and hence by Lemma~\ref{L-free7} we do have $s_1x''\in D_{n_1}\times C_{n_1}$.

 Next consider the subcase that for any $\tilde{x}\in X$ such that $s_1\tilde{x}\in D_{n_1}\times C_{n_1}$ and $s_2\tilde{x}\in D_{m}\times C_{m}$ we have $m=n_2$.  If $t$ starts with $u_{n_1}$, then for any large enough $m>n_1$ and any $\tilde{x}\in X$ with $s_2\tilde{x}\in D_{m}\times C_{m}$ one has
$s_1\tilde{x}\in D_{n_1}\times C_{n_1}$, which contradicts our assumption. Since $tD_{n_2}\cap D_{n_1}$ is nonempty,
by Lemma~\ref{L-free2} either
 $t=u_{n_1}u_{n_2}^{-1}$ or $t$ ends with $u_{n_2}^{-1}$.
Suppose that $t=u_{n_1}u_{n_2}^{-1}$. We say $(s_1, s_2)$ has type (B3). Clearly $n_1$ and $n_2$ are determined by $t$.
Let $x''=(y'', z'')\in X$ with $s_2x''\in D_{k_2}\times C_{k_2}$. If $s_1x''\in X_+$, then by \eqref{E-B} and our assumption $k_2=n_2$ and  $s_1x''\in D_{n_1}\times C_{n_1}$. Conversely, if $k_2=n_2$, then by Lemma~\ref{L-free6} we do have $s_1x''\in D_{n_1}\times C_{n_1}$. Now suppose that $t$ ends with $u_{n_2}^{-1}$ instead. Then $t=vu_{n_2}^{-1}$ for some $v\in \Gamma$ not ending with $b^{-1}$ and not starting with $u_{n_1}$. We say $(s_1, s_2)$ has type (B4). Clearly $n_2$ is determined by $t$. Also $n_1$ is determined by $t$ as $t\in C_{n_1}$.
Let $x''=(y'', z'')\in X$ with $s_2x''\in D_{k_2}\times C_{k_2}$.
If $s_1x''\in X_+$, then by \eqref{E-B} and our assumption $k_2=n_2$ and  $s_1x''\in D_{n_1}\times C_{n_1}$, and hence by Lemma~\ref{L-free6} the point $s_2y''$ starts with $t^{-1}u_{n_1}=u_{n_2}v^{-1}u_{n_1}$. Conversely, if $k_2=n_2$ and $s_2y''$ starts with $t^{-1}u_{n_1}$, then by Lemma~\ref{L-free6} we do have $s_1x''\in D_{n_1}\times C_{n_1}$.

Finally consider the case there is no $x\in X$ such that $s_1x\in D_{n_1}\times C_{n_1}$, $s_2x\in D_{n_2}\times C_{n_2}$ with $n_1\neq n_2$. We separate it into three subcases.

Consider first the subcase that there are $x, x'\in X$ such that $s_1x, s_2x\in D_{n_2}\times C_{n_2}$ and $s_1x', s_2x'\in D_{n_1}\times C_{n_1}$ with $n_1\neq n_2$. Since $tD_{n_2}\cap D_{n_2}$ and $tD_{n_1}\cap D_{n_1}$ are nonempty, by Lemma~\ref{L-free2} the element $t$ must end with either $u_{n_1}^{-1}$ or $u_{n_2}^{-1}$.
Without loss of generality, assume that $t$ ends with $u_{n_2}^{-1}$. As $tD_{n_1}\cap D_{n_1}$ is nonempty, Lemma~\ref{L-free2} implies that $t$ starts with $u_{n_1}$.
Thus $t=u_{n_1}vu_{n_2}^{-1}$  for some nontrivial $v\in \Gamma$ not starting with $b$ and not ending with $b^{-1}$.
We say $(s_1, s_2)$ has type (C1). Clearly $n_1$ and $n_2$ are determined by $t$.
Let $x''=(y'', z'')\in X$ such that $s_2x''\in D_{k_2}\times C_{k_2}$.
If $k_2\neq n_2$, then $s_1y''=t(s_2y'')\in D_{n_1}$. Thus if $k_2\neq n_2$ and $s_1x''\in X_+$ then $s_1x''\in D_{n_1}\times C_{n_1}$ and hence $k_2=n_1$ by Lemma~\ref{L-free5}. Conversely, if $k_2=n_1$, then by Lemma~\ref{L-free5} we do have $s_1z''\in C_{n_1}$ and hence $s_1x''\in D_{n_1}\times C_{n_1}$.
If $k_2=n_2$, then by Lemma~\ref{L-free5} we have $s_1z''\in C_{n_2}$. Thus if $k_2=n_2$
 and $s_1x''\in X_+$, then $s_1x''\in D_{n_2}\times C_{n_2}$ and hence  by Lemma~\ref{L-free6} the point $s_2y''$ must start with $t^{-1}u_{n_2}=u_{n_2}v^{-1}u_{n_1}^{-1}u_{n_2}$. Conversely, if $k_2=n_2$ and $s_2y''$ starts with $t^{-1}u_{n_2}$, then by Lemma~\ref{L-free6} we do have $s_1x''\in D_{n_2}\times C_{n_2}$.

Next consider the subcase that there is some $m\ge 2$ such that for any $x=(y, z)\in X$ such that $s_1x\in D_{n_1}\times C_{n_1}$, $s_2x\in D_{n_2}\times C_{n_2}$ one has $n_1=n_2=m$. Since $tD_m\cap D_m$ is nonempty, by Lemma~\ref{L-free2} either $t$ starts with $u_m$ or $t$ ends with $u_m^{-1}$. Note that $t$ starts with $u_m$ exactly when $t^{-1}$ ends with $u_m^{-1}$. By symmetry we may assume that $t$ ends with $u_m^{-1}$. Suppose that $t$ also starts with $u_m$. Then $t=u_mvu_m^{-1}$ for some nontrivial $v\in \Gamma$ not starting with $b$ and not ending with $b^{-1}$. We say that $(s_1, s_2)$ has type (C2). In this case $m$ is determined by $t$.
Let $x''=(y'', z'')\in X$ such that $s_2x''\in D_{k_2}\times C_{k_2}$.
If $s_1x''\in X_+$, then by our assumption $k_2=m$ and $s_1x''\in D_m\times C_m$, and hence by Lemma~\ref{L-free7} the point $s_2y''$ does not start with $t^{-1}u_mb= u_mv^{-1}b$. Conversely, if $k_2=m$ and $s_2y''$ does not start with $t^{-1}u_mb$, then by Lemmas~\ref{L-free7} and \ref{L-free5} we do have $s_1x''\in D_m\times C_m$. Now suppose instead that $t$ does not start with $u_m$. Then $t=vu_m^{-1}$ for some $v\in \Gamma$ not starting with $u_m$ and not ending with $b^{-1}$.
We say that $(s_1, s_2)$ has type (C3). In this case $m$ is also determined by $t$.
Let $x''=(y'', z'')\in X$ such that $s_2x''\in D_{k_2}\times C_{k_2}$.
If $s_1x''\in X_+$, then by our assumption $k_2=m$ and $s_1x''\in D_m\times C_m$, and hence by Lemma~\ref{L-free6} the point $s_2y''$ starts with $t^{-1}u_m=u_mv^{-1}u_m$.
Conversely, if $k_2=m$ and $s_2y''$ starts with $t^{-1}u_m$, then by Lemmas~\ref{L-free6} and \ref{L-free5} we do have $s_1x''\in D_m\times C_m$.
\end{proof}

In the next $9$ lemmas we give an upper bound for the size of an independence set $M\subseteq \Gamma$ for $(X_+, X_-)$ when every pair in $M$ (with respect to some linear order of $M$) has a fixed type in Lemma~\ref{L-free3}.

For each finite independence set $M\subseteq \Gamma$  for $(X_+, X_-)$ and  each map $\omega: M\rightarrow \{+, -\}$, fix
a point $x_\omega=(y_\omega, z_\omega)\in \bigcap_{s\in M}s^{-1}X_{\omega(s)}$. For any such $\omega$ and any $s\in \omega^{-1}(+)$, the point $sx_{\omega}$ lies in $D_n\times C_n$ for a unique $n\ge 2$. Denote this $n$ by $g(\omega, s)$.

\begin{lemma} \label{L-A1}
Let $M\subseteq \Gamma$ be a finite independence set for $(X_+, X_-)$ with cardinality $\ell$.
List the elements of $M$ as $s'_1, \dots, s'_{\ell}$. Assume that $(s'_i, s'_j)$ has type (A1) for all $1\le i<j\le \ell$. Then $\ell<n_{A1}:=21$.
\end{lemma}
\begin{proof} Assume that $\ell\ge 21$.

We claim that there are some map $\omega_0: M\rightarrow \{+, -\}$ and a set $A\subseteq M$ with cardinality $5$
such that $\omega_0=+$ on $A$ and the map $s\mapsto g(\omega_0, s)$ on $A$ is injective. Let $\omega'$ be the map $M\rightarrow \{+\}$. Then either there is some subset $A$ of $M$ with cardinality $5$ such that the map $s\mapsto g(\omega', s)$ on $A$ is injective  or there is some subset $B$ of $M$ with cardinality $6$ such that $g(\omega', s)=m$ for some $m\ge 2$ and all $s\in B$. In the first situation we can take $\omega_0=\omega'$. Thus assume that $B\subseteq M$ has cardinality $6$ and $g(\omega', s)=m$ for  all $s\in B$. List the elements of $B$ as $\theta_1, \dots, \theta_6$ such that $(\theta_i, \theta_j)$ has type (A1) for all $1\le i<j\le 6$. Take $\omega'': M\rightarrow \{+, -\}$ such that $\omega''(\theta_1)=-$ and $\omega''(\theta_i)=+$ for all $2\le i\le 6$. For any $2\le i\le 6$, since $(\theta_1, \theta_i)$ has type (A1) and $\omega''(\theta_1)=-$, we have $g(\omega'', \theta_i)\neq m$. For any $2\le i<j\le 6$, since $(\theta_i, \theta_j)$ has type (A1), we have $g(\omega'', \theta_i)<g(\omega'', \theta_j)$. Then we can take $A=\{\theta_2, \dots, \theta_6\}$ and $\omega_0=\omega''$.
This proves our claim.

List the elements of $A$ as $s_1, \dots, s_5$ such that $(s_i, s_j)$ has type (A1) for all $1\le i<j\le 5$. Set $n_i=g(\omega_0, s_i)$. Then $n_1<n_2<\cdots<n_5$. Now $s_5y_{\omega_0}$ starts with $\xi_i:=s_5s_i^{-1}u_{n_i}$ for $i=2, 3, 4$.
Write $\{2, 3, 4\}$ as $\{i_1, i_2, i_3\}$ such that the length of $\xi_{i_3}$  is no less than those of $\xi_{i_2}$ and $\xi_{i_1}$.

For $k=1, 2$, take $\omega_k: M\rightarrow \{+, -\}$ such that $\omega_k(s_{i_k})=-$ and $\omega_k(s)=+$ for all $s\in A\setminus \{s_{i_k}\}$. We claim that $g(\omega_k, s_5)<n_1$. Suppose that $g(\omega_k, s_5)\ge n_1$ instead.
Since $(s_1, s_5)$ has type (A1), we have $g(\omega_k, s_5)=n_5$. As $(s_{i_3}, s_5)$ has type (A1), the element $s_5y_{\omega_k}$ starts with $\xi_{i_3}$ and hence starts with $\xi_{i_k}$. Since $(s_{i_k}, s_5)$ has type (A1), we get $s_{i_k}x_{\omega_k}\in X_+$, which is a contradiction to $\omega_k(s_{i_k})=-$. This proves our claim.

Since $(s_{i_1}, s_5)$ has type (A1) and $g(\omega_2, s_5)<n_1<n_{i_1}$, we get $g(\omega_2, s_{i_1})=g(\omega_2, s_5)$.
As $(s_1, s_5)$ has type (A1) and $g(\omega_1, s_5), g(\omega_2, s_5)<n_1$, we have $g(\omega_1, s_5)=g(\omega_2, s_5)$.
Since $(s_{i_1}, s_5)$ has type (A1) and $g(\omega_1, s_5)=g(\omega_2, s_5)=g(\omega_2, s_{i_1})<n_{i_1}$, we get $s_{i_1}x_{\omega_1}\in X_+$, which is a contradiction to $\omega_1(s_{i_1})=-$.
\end{proof}

\begin{lemma} \label{L-A2}
Let $M\subseteq \Gamma$ be a finite independence set for $(X_+, X_-)$ with cardinality $\ell$.
List the elements of $M$ as $s'_1, \dots, s'_{\ell}$. Assume that $(s'_i, s'_j)$ has type (A2) for all $1\le i<j\le \ell$. Then $\ell< n_{A2}:=13$.
\end{lemma}
\begin{proof} Assume that $\ell\ge 13$.

We claim that there are some map $\omega_0: M\rightarrow \{+, -\}$ and a set $A\subseteq M$ with cardinality $4$
such that $\omega_0=+$ on $A$ and the map $s\mapsto g(\omega_0, s)$ on $A$ is injective. Let $\omega'$ be the map $M\rightarrow \{+\}$. Then either there is some subset $A$ of $M$ with cardinality $4$ such that the map $s\mapsto g(\omega', s)$ on $A$ is injective  or there is some subset $B$ of $M$ with cardinality $5$ such that $g(\omega', s)=k$ for some $k\ge 2$ and all $s\in B$. In the first situation we can take $\omega_0=\omega'$. Thus assume that $B\subseteq M$ has cardinality $5$ and $g(\omega', s)=k$ for  all $s\in B$.
Take $\theta_5\in B$ such that $(\theta, \theta_5)$ has type (A2) for all $\theta
\in B\setminus \{\theta_5\}$. Note that $\theta_5y_{\omega'}$ starts with $\zeta_\theta:=\theta_5\theta^{-1}u_k$ for every $\theta\in B\setminus \{\theta_5\}$.
Take $\theta_1\in B\setminus \{\theta_5\}$ such that the length of $\zeta_{\theta_1}$ is no bigger than that of $\zeta_{\theta}$ for all $\theta\in B\setminus \{\theta_1, \theta_5\}$.
Take $\omega'': M\rightarrow \{+, -\}$ such that $\omega''(\theta_1)=-$ and $\omega''(\theta)=+$ for all $\theta\in B\setminus \{\theta_1\}$.
For any $\theta\in B\setminus \{\theta_1, \theta_5\}$, since $(\theta, \theta_5)$ has type (A2), either $k<g(\omega'', \theta)<g(\omega'', \theta_5)$ or $k=g(\omega'', \theta)=g(\omega'', \theta_5)$ and $\theta_5y_{\omega''}$ starts with $\zeta_{\theta}$. In the latter situation $\theta_5y_{\omega''}$ also starts with $\zeta_{\theta_1}$, and hence $\theta_1x_{\omega''}\in X_+$, which contradicts $\omega''(\theta_1)=-$. Thus  $k<g(\omega'', \theta)<g(\omega'', \theta_5)$ for every $\theta\in B\setminus \{\theta_1, \theta_5\}$. For any distinct $\theta, \theta'\in B\setminus \{\theta_1, \theta_5\}$, since either $(\theta, \theta')$ or $(\theta', \theta)$ has type (A2), we have $g(\omega'', \theta)\neq g(\omega'', \theta')$.
Then we can take $A=B\setminus \{\theta_1\}$ and $\omega_0=\omega''$.
This proves our claim.

List the elements of $A$ as $s_1, \dots, s_4$ such that $(s_i, s_j)$ has type (A2) for all $1\le i<j\le 4$. Set $n_i=g(\omega_0, s_i)$. Then $n_1<n_2<n_3<n_4$.

Take $\omega_1: M\rightarrow \{+, -\}$ such that $\omega_1(s_1)=-$ and $\omega_1(s)=+$ for all $s\in A\setminus \{s_1\}$. Since $(s_1, s_4)$ has type (A2), we have $g(\omega_1, s_4)\le n_1$. Set $m=g(\omega_1, s_4)$. For $i=2, 3$, since $(s_i, s_4)$ has type (A2) and $g(\omega_1, s_4)\le n_1<n_i$, we get that $g(\omega_1, s_i)=m$ and  $s_4y_{\omega_1}$ starts with $\xi_i:=s_4s_i^{-1}u_m$.
Write $\{2, 3\}$ as $\{i_1, i_2\}$ such that the length of $\xi_{i_2}$ is no less than that of $\xi_{i_1}$.

Take $\omega_2: M\rightarrow \{+, -\}$ such that $\omega_2(s_1)=\omega_2(s_{i_1})=-$ and $\omega_2(s_{i_2})=\omega_2(s_4)=+$. Since $(s_1, s_4)$ has type (A2), we have $g(\omega_2, s_4)\le n_1$. As $(s_{i_2}, s_4)$ has type (A2) and $g(\omega_2, s_4)\le n_1<n_{i_2}$, we have  $g(\omega_2, s_4)=m$ and $s_4y_{\omega_2}$ starts with $\xi_{i_2}$. Then  $s_4y_{\omega_2}$ also starts with $\xi_{i_1}$. Since $(s_{i_1}, s_4)$ has type (A2) and $g(\omega_2, s_4)=m=g(\omega_1, s_{i_1})$, we get that $s_{i_1}x_{\omega_2}\in X_+$, which contradicts $\omega_2(s_{i_1})=-$.
\end{proof}

\begin{lemma} \label{L-B1}
Let $M\subseteq \Gamma$ be a finite independence set for $(X_+, X_-)$ with cardinality $\ell$.
List the elements of $M$ as $s_1, \dots, s_{\ell}$. Assume that $(s_i, s_j)$ has type (B1) for all $1\le i<j\le \ell$. Then $\ell<n_{B1}:=4$.
\end{lemma}
\begin{proof} Assume that $\ell=4$.

Let $\omega_0$ be the map $M\rightarrow \{+\}$. Set $n_i=g(\omega_0, s_i)$. Then $n_1<n_2<n_3<n_4$.

Define $\omega_1: M\rightarrow \{+, -\}$ by $\omega_1(s_1)=\omega_1(s_2)=-$ and $\omega_1(s_3)=\omega_1(s_4)=+$. Since $(s_3, s_4)$ has type (B1), we have $g(\omega_1, s_4)>n_3$. Set $m=g(\omega_1, s_4)$. For $i=1, 2$, since $(s_i, s_4)$ has type (B1) and $g(\omega_1, s_4)>n_3>n_i$, we get that  $s_4y_{\omega_1}$ starts with $\xi_i:=s_4s_i^{-1}u_{n_i}b$. Write $\{1, 2\}$ as $\{i_1, i_2\}$ such that the length of $\xi_{i_2}$ is no less than that of $\xi_{i_1}$.

 Define $\omega_2: M\rightarrow \{+, -\}$ by $\omega_2(s_{i_2})=-$ and $\omega_2(s)=+$ for all $s\in M\setminus \{s_{i_2}\}$. Since $(s_3, s_4)$ has type (B1), we have $g(\omega_2, s_4)>n_3$. As $(s_{i_2}, s_4)$ has type (B1), we have  $g(\omega_2, s_4)=m$ and $s_4y_{\omega_2}$ starts with $\xi_{i_2}$. Then  $s_4y_{\omega_2}$ also starts with $\xi_{i_1}$. Since $(s_{i_1}, s_4)$ has type (B1), we get that $s_{i_1}x_{\omega_2}\not\in X_+$, which contradicts $\omega_2(s_{i_1})=+$.
\end{proof}

\begin{lemma} \label{L-B2}
Let $M\subseteq \Gamma$ be a finite independence set for $(X_+, X_-)$ with cardinality $\ell$.
List the elements of $M$ as $s_1, \dots, s_{\ell}$. Assume that $(s_i, s_j)$ has type (B2) for all $1\le i<j\le \ell$. Then $\ell<n_{B2}:=4$.
\end{lemma}
\begin{proof} Assume that $\ell=4$.

Let $\omega_0$ be the map $M\rightarrow \{+\}$. Set $n_i=g(\omega_0, s_i)$. Then $n_1<n_2<n_3<n_4$.

Define $\omega_1: M\rightarrow \{+, -\}$ by $\omega_1(s_1)=\omega_1(s_2)=-$ and $\omega_1(s_3)=\omega_1(s_4)=+$. Since $(s_3, s_4)$ has type (B2), we have $g(\omega_1, s_4)>n_3$. Set $m=g(\omega_1, s_4)$. For $i=1, 2$, since $(s_i, s_4)$ has type (B2) and $g(\omega_1, s_4)>n_3>n_i$, we get that  $s_is_4^{-1}=u_{n_i}a^{-m}$.

 Define $\omega_2: M\rightarrow \{+, -\}$ by $\omega_2(s_2)=-$ and $\omega_2(s)=+$ for all $s\in M\setminus \{s_2\}$. Since $(s_3, s_4)$ has type (B2), we have $g(\omega_2, s_4)>n_3$. Set $m'=g(\omega_2, s_4)$. As $(s_2, s_4)$ has type (B2) and $g(\omega_2, s_4)>n_3>n_2$, we have  $s_2s_4^{-1}=u_{n_2}a^{-m'}$, and hence $m=m'$.
 Then $s_1s_4^{-1}=u_{n_1}a^{-m'}$.
 Since $(s_1, s_4)$ has type (B2) and $s_1s_4^{-1}=u_{n_1}a^{-m'}$, we get that $s_1x_{\omega_2}\not\in X_+$, which contradicts $\omega_2(s_1)=+$.
\end{proof}

\begin{lemma} \label{L-B3}
Let $M\subseteq \Gamma$ be a finite independence set for $(X_+, X_-)$ with cardinality $\ell$.
List the elements of $M$ as $s_1, \dots, s_{\ell}$. Assume that $(s_i, s_j)$ has type (B3) for all $1\le i<j\le \ell$. Then $\ell<n_{B3}:=3$.
\end{lemma}
\begin{proof} Assume that $\ell=3$.

Let $\omega_0$ be the map $M\rightarrow \{+\}$. Set $n_i=g(\omega_0, s_i)$. Then $n_1<n_2<n_3$.

Define $\omega_1: M\rightarrow \{+, -\}$ by $\omega_1(s_1)=-$ and $\omega_1(s_2)=\omega_1(s_3)=+$. Since $(s_2, s_3)$ has type (B3), we have $g(\omega_1, s_3)=g(\omega_0, s_3)$.
As $(s_1, s_3)$ has type (B3) and $g(\omega_1, s_3)=g(\omega_0, s_3)$, we get that
$s_1x_{\omega_1}\in X_+$, which contradicts $\omega_1(s_1)=-$.
\end{proof}

\begin{lemma} \label{L-B4}
Let $M\subseteq \Gamma$ be a finite independence set for $(X_+, X_-)$ with cardinality $\ell$.
List the elements of $M$ as $s_1, \dots, s_{\ell}$. Assume that $(s_i, s_j)$ has type (B4) for all $1\le i<j\le \ell$. Then $\ell<n_{B4}:=3$.
\end{lemma}
\begin{proof} Assume that $\ell=3$.

Let $\omega_0$ be the map $M\rightarrow \{+\}$. Set $n_i=g(\omega_0, s_i)$. Then $n_1<n_2<n_3$. For $i=1, 2$, since $(s_i, s_3)$ has type (B4), we get that $s_3y_{\omega_0}$ starts with $\xi_i:=s_3s_i^{-1}u_{n_i}$.
Write $\{1, 2\}$ as $\{i_1, i_2\}$ such that the length of $\xi_{i_2}$ is no less than that of $\xi_{i_1}$.

Define $\omega_1: M\rightarrow \{+, -\}$ by $\omega_1(s_{i_1})=-$ and $\omega_1(s_{i_2})=\omega_1(s_3)=+$. Since $(s_{i_2}, s_3)$ has type (B4), we have that $g(\omega_1, s_3)=g(\omega_0, s_3)$ and
$s_3y_{\omega_1}$ starts with $\xi_{i_2}$. Then $s_3y_{\omega_1}$ also starts with $\xi_{i_1}$.
 Since $(s_{i_1}, s_3)$ has type (B4), we get that $s_{i_1}x_{\omega_1}\in X_+$, which contradicts $\omega_1(s_{i_1})=-$.
\end{proof}

\begin{lemma} \label{L-C1}
Let $M\subseteq \Gamma$ be a finite independence set for $(X_+, X_-)$ with cardinality $\ell$.
List the elements of $M$ as $s'_1, \dots, s'_{\ell}$. Assume that $(s'_i, s'_j)$ has type (C1) for all $1\le i<j\le \ell$.
Then $\ell<n_{C1}:=6$.
\end{lemma}
\begin{proof} Assume that $\ell=6$.

Let $\omega_0$ be the map $M\rightarrow \{+\}$. Then there is some $m\ge 2$ such that $g(\omega_0, s)=m$ for all $s\in M$.
Denote by $A$ the set of $s\in M\setminus \{s'_6\}$ such that $sx\in D_m\times C_m$ whenever $x\in X$ and $s'_6x\in D_m\times C_m$.
Set $B=M\setminus (\{s'_6\}\cup A)$. Then either $|A|\ge 3$ or $|B|\ge 3$.

Consider first the case $|A|\ge 3$. Take distinct  points $s_1, s_2, s_3\in A$. Define a map $\omega': M\rightarrow \{+, -\}$ by $\omega'(s_3)=-$ and $\omega'(s)=+$ for all $s\in M\setminus \{s_3\}$.
Then there is some $n\ge 2$ such that $g(\omega', s)=n$ for all $s\in M\setminus \{s_3\}$.  Since $s_3\in A$, we have $n\neq m$.
For $i=1, 2$, since $(s_i, s'_6)$ has type (C1), we get that $s'_6y_{\omega'}$ starts with $\xi_i:=s'_6s_i^{-1}u_n$.
Write $\{1, 2\}$ as $\{i_1, i_2\}$ such that the length of $\xi_{i_2}$ is no less than that of $\xi_{i_1}$.
Take a map $\omega_1: M\rightarrow \{+, -\}$ such that $\omega_1(s_3)=\omega_1(s_{i_1})=-$ and $\omega_1(s_{i_2})=\omega(s'_6)=+$.
 As $(s_{i_2}, s'_6)$ has type (C1), we know that $g(\omega_1, s_{i_2})=g(\omega_1, s'_6)$ must be either $n$ or $m$. As $\omega_1(s_3)=-$, we have $g(\omega_1, s_{i_2})=g(\omega_1, s'_6)=n$, and hence
$s'_6y_{\omega_1}$ starts with $\xi_{i_2}$. Then $s'_6y_{\omega_1}$ also starts with $\xi_{i_1}$.
Since $(s_{i_1}, s'_6)$ has type (C1), it follows that $s_{i_1}x_{\omega_1}\in X_+$, which contradicts $\omega_1(s_{i_1})=-$.

 Next consider the case $|B|\ge 3$. Take distinct  points $s_1, s_2, s_3\in B$. For any $i=1, 2$, we have that $s'_6y_{\omega_0}$ starts with $\xi_i:=s'_6s_i^{-1}u_m$.
 Write $\{1, 2\}$ as $\{i_1, i_2\}$ such that the length of $\xi_{i_2}$ is no less than that of $\xi_{i_1}$. Consider any map $\omega':M\rightarrow \{+, -\}$ satisfying that
 $\omega'(s_{i_1})=-$ and $\omega'(s_{i_2})=\omega'(s'_6)=+$. Since $(s_{i_2}, s'_6)$ has type (C1), if $g(\omega', s_{i_2})=g(\omega', s'_6)=m$, then $s'_6y_{\omega'}$ starts with $\xi_{i_2}$ and hence starts with $\xi_{i_1}$, which implies that $s_{i_1}x_{\omega'}\in X_+$, a contradiction. Therefore $g(\omega', s_{i_2})=g(\omega', s'_6)$ is different from $m$, and hence does not depend on the choice of $\omega'$. Take two maps $\omega_1, \omega_2:M\rightarrow \{+, -\}$ satisfying the conditions for $\omega'$ such that $\omega_1(s_3)=+$ while $\omega_2(s_3)=-$. Set $n=g(\omega_1, s_{i_2})=g(\omega_1, s'_6)=g(\omega_1, s_3)$. Then $n\neq m$, and $n=g(\omega_2, s_{i_2})=g(\omega_2, s'_6)$. Since $(s_3, s'_6)$ has type (C1) and $s_3\not\in A$, for any $x\in X$ with $s'_6x\in D_n\times C_n$ we have $s_3x\in D_n\times C_n$. In particular,  $s_3x_{\omega_2}\in D_n\times C_n$, which contradicts $\omega_2(s_3)=-$.
\end{proof}

\begin{lemma} \label{L-C2}
Let $M\subseteq \Gamma$ be a finite independence set for $(X_+, X_-)$ with cardinality $\ell$.
List the elements of $M$ as $s_1, \dots, s_{\ell}$. Assume that $(s_i, s_j)$ has type (C2) for all $1\le i<j\le \ell$. Then $\ell<n_{C2}:=4$.
\end{lemma}
\begin{proof} Assume that $\ell=4$.

Let $\omega_0$ be the map $M\rightarrow \{+\}$. Then there is some $m\ge 2$ such that $g(\omega_0, s)=m$ for all $s\in M$.

Define $\omega_1: M\rightarrow \{+, -\}$ by $\omega_1(s_1)=\omega(s_2)=-$ and $\omega_1(s_3)=\omega_1(s_4)=+$. Since $(s_3, s_4)$ has type (C2), we have $g(\omega_1, s_4)=m$.
For $i=1, 2$, as $(s_i, s_4)$ has type (C2), we know
that
$s_4y_{\omega_1}$ starts with $\xi_i:=s_4s_i^{-1}u_mb$.
Write $\{1, 2\}$ as $\{i_1, i_2\}$ such that the length of $\xi_{i_2}$ is no less than that of $\xi_{i_1}$.

Consider any map $\omega_2: M\rightarrow \{+, -\}$ satisfying $\omega_2(s_{i_2})=-$ and $\omega_2(s_{i_1})=\omega_2(s_4)=+$. Since $(s_{i_1}, s_4)$ has type (C2), we have that $g(\omega_2, s_4)=m$ and
$s_4y_{\omega_2}$ does not start with $\xi_{i_1}$. Then $s_4y_{\omega_2}$ does not start with $\xi_{i_2}$.
 As $(s_{i_2}, s_4)$ has type (C2), we get that $s_{i_2}x_{\omega_2}\in X_+$, which contradicts $\omega_2(s_{i_2})=-$.
\end{proof}

\begin{lemma} \label{L-C3}
Let $M\subseteq \Gamma$ be a finite independence set for $(X_+, X_-)$ with cardinality $\ell$.
List the elements of $M$ as $s_1, \dots, s_{\ell}$. Assume that $(s_i, s_j)$ has type (C3) for all $1\le i<j\le \ell$. Then $\ell<n_{C3}:=3$.
\end{lemma}
\begin{proof} Assume that $\ell=3$.

Let $\omega_0$ be the map $M\rightarrow \{+\}$. Then there is some $m\ge 2$ such that $g(\omega_0, s)=m$ for all $s\in M$.
For $i=1, 2$, since $(s_i, s_3)$ has type (C3), we get that $s_3y_{\omega_0}$ starts with $\xi_i:=s_3s_i^{-1}u_m$.
Write $\{1, 2\}$ as $\{i_1, i_2\}$ such that the length of $\xi_{i_2}$ is no less than that of $\xi_{i_1}$.

Define $\omega_1: M\rightarrow \{+, -\}$ by $\omega_1(s_{i_1})=-$ and $\omega_1(s_{i_2})=\omega_1(s_3)=+$. Since $(s_{i_2}, s_3)$ has type (C3), we have that $g(\omega_1, s_3)=m$ and
$s_3y_{\omega_1}$  starts with $\xi_{i_2}$. Then $s_3y_{\omega_1}$ starts with $\xi_{i_1}$.
 As $(s_{i_1}, s_3)$ has type (C3), we get that $s_{i_1}x_{\omega_1}\in X_+$, which contradicts $\omega_1(s_{i_1})=-$.
\end{proof}

According to the Ramsey theorem \cite[page 183]{Bollobas}, given any natural numbers $k$ and $c_1, \dots, c_k$ there is a natural number $n$ such that for any graph $\cG$ with $n$ vertices and exactly one (unoriented) edge between any two distinct vertices, if we color the edges of $\cG$ into $k$ colors, then there are some $1\le i\le k$ and  a set $A$ of the vertices of $\cG$ with cardinality $c_i$ such that the edge between any two distinct vertices in $A$  has the $i$-th color. The smallest such number $n$ is denoted by $R_k(c_1, \dots, c_k)$.

\begin{lemma} \label{L-free4}
Each independence set $M\subseteq \Gamma$ for $(X_+, X_-)$ has cardinality strictly less than
$$R_{18}(n_{A1}, n_{A1}, n_{A2}, n_{A2}, n_{B_1}, n_{B_1}, n_{B2}, n_{B2}, n_{B3}, n_{B3}, n_{B4}, n_{B4}, n_{C1}, n_{C1}, n_{C2}, n_{C2}, n_{C3}, n_{C3}).$$
\end{lemma}
\begin{proof} We may assume that $M$ is finite. Consider the graph $\cG$ with vertex set $M$ and one unoriented edge between any two distinct vertices.
We shall color the edges of $\cG$ with $18$ colors as follows.
Fix a linear order on $M$. For any $s<s'$ in $M$, we color the edge between $s$ and $s'$ with color $*$ (resp. $*'$) if the pair $(s, s')$ (resp. $(s', s)$) has type $*$, where $*$ is one of the $9$ types (A1), $\dots$, (C3) in Lemma~\ref{L-free3}.
If certain edge can be colored in more than one way, take any choice.
If $|M|$ is no less than the above Ramsey number, then for certain type $*$, there is a subset $A$ of $M$ with cardinality $n_*$ such that either the pair $(s, s')$ has type $*$ for all $s<s'$ in $A$ or the pair $(s', s)$ has type $*$ for all $s<s'$ in $A$, which is impossible by Lemmas~\ref{L-A1}-\ref{L-C3}.
\end{proof}

By \cite[Corollary 12.3]{KL07} the action $\Gamma\curvearrowright Y$ is null. From Lemma~\ref{L-equicontinuous null}
we know that $\Gamma\curvearrowright Z$ is null.
It is also clear from (2) and (4) of Proposition~\ref{P-null} that the class of null actions is closed under taking products. Thus the product action $\Gamma\curvearrowright X$ is null. Then from Proposition~\ref{P-construction null} and Lemma~\ref{L-free4} we conclude that the action $\Gamma\curvearrowright X_f$ is null. This finishes the proof of Theorem~\ref{T-main}.

To end this section, we discuss one observation pointed out to us by Glasner. A minimal action $\Gamma\curvearrowright \cX$ is called {\it strictly SPI} \cite[page 127]{Glasner87} if there are an ordinal $\tau$, an action $\Gamma\curvearrowright \cX_\alpha$ and a factor map $\pi_\alpha: \cX\rightarrow \cX_\alpha$ for each ordinal $\alpha\le \tau$ such that
$\pi_\tau$ is the identity map of $\cX$, $\cX_0$ is a singleton, $\pi_\beta$ factors through $\pi_\alpha$ for all $\beta<\alpha\le \tau$, for each ordinal $\alpha<\tau$ the factor map $\cX_{\alpha +1}\rightarrow \cX_\alpha$ is either strongly proximal or isometric, and for each limit ordinal $\alpha\le \tau$ the map $\pi_\alpha$ is the inverse limit of $\pi_\beta$ for $\beta<\alpha$. In such case, one has a canonical choice of these $\Gamma\curvearrowright \cX_\alpha$, called the {\it canonical SPI tower} of $\Gamma\curvearrowright \cX$, defined inductively by taking $\Gamma\curvearrowright X_0$ to be the action on a singleton, and for each limit ordinal $\alpha$ taking
$\Gamma\curvearrowright \cX_\alpha$ to be the inverse limit of $\Gamma\curvearrowright \cX_\beta$ for all $\beta<\alpha$, taking $\Gamma\curvearrowright \cX_{\alpha+2n+1}$ to
be the largest strongly proximal factor  of $\cX\rightarrow \cX_{\alpha+2n}$ (i.e. one has factor maps $\cX\rightarrow \cX_{\alpha+2n+1}\rightarrow \cX_{\alpha+2n}$, the extension $\cX_{\alpha+2n+1}\rightarrow \cX_{\alpha+2n}$ is strongly proximal, and for any factor maps $\cX\rightarrow \cX'\rightarrow \cX_{\alpha+2n}$ with $\cX'\rightarrow \cX_{\alpha+2n}$ strongly proximal the map $\cX\rightarrow \cX'$ factors through $\cX\rightarrow \cX_{\alpha+2n+1}$) for all integers $n\ge 0$, and taking $\Gamma\curvearrowright \cX_{\alpha+2n+2}$ to
be the largest isometric factor  of $\cX\rightarrow \cX_{\alpha+2n+1}$ (i.e. one has factor maps $\cX\rightarrow \cX_{\alpha+2n+2}\rightarrow \cX_{\alpha+2n+1}$, the extension $\cX_{\alpha+2n+2}\rightarrow \cX_{\alpha+2n+1}$ is isometric, and for any factor maps $\cX\rightarrow \cX'\rightarrow \cX_{\alpha+2n+1}$ with $\cX'\rightarrow \cX_{\alpha+2n+1}$ isometric the map $\cX\rightarrow \cX'$ factors through $\cX\rightarrow \cX_{\alpha+2n+2}$) for all integers $n\ge 0$. This inductive process stops at some ordinal $\tau$ when the map $\cX\rightarrow \cX_\tau$ is a homeomorphism. We are grateful to Glasner for showing us the following result.

\begin{proposition} \label{P-tower}
The action $\Gamma\curvearrowright X_f$ is strictly SPI, and $X_f\rightarrow X\rightarrow Y\rightarrow \{pt\}$ is its canonical SPI tower, where $\Gamma\curvearrowright \{pt\}$ is the trivial action on a singleton.
\end{proposition}
\begin{proof} We know that $\Gamma\curvearrowright Y$ is strongly proximal. Let $\Gamma\curvearrowright Y^*$ be the largest strongly proximal factor of $\Gamma\curvearrowright X_f$. Then we have factor maps $X_f\rightarrow Y^*\overset{\vartheta}\rightarrow Y$. The proof of Lemma~\ref{L-minimal} shows that the product action $\Gamma\curvearrowright Y^*\times Z$ is minimal, and hence it is a factor of $\Gamma\curvearrowright X_f$. We now have the following diagram:
\begin{equation*}
\xymatrix
{
 X_f   \ar[d] \ar@/^2pc/@{>}^{\pi_f}[dd]\\
   Y^*\times Z \ar[d]_{\vartheta\times \id_Z}\\
 Y\times Z
}
\end{equation*}
By construction, the map $\pi_f$ has the property that $\pi_f^{-1}(y, z)$ is a singleton on the complement of a countable set. If for some $y\in Y$ the set $\vartheta^{-1}(y)$ is not a singleton, then for each $z\in Z$ we have that $(\vartheta\times \id_Z)^{-1}(y, z)=\vartheta^{-1}(y)\times \{z\}$ is not a singleton as well, which is a contradiction since $Z$ is uncountable.  Thus $\vartheta$ is injective, and hence $\Gamma\curvearrowright Y$ is the largest strongly proximal factor of $\Gamma\curvearrowright X_f$.

The extension $X=Y\times Z\rightarrow Y$ is clearly isometric.
The extension $\pi_f$ is almost one-to-one, hence has no nontrivial isometric factor.
It follows that $\Gamma\curvearrowright X$ is the largest isometric factor of $X_f\rightarrow Y$.

Finally, since $\pi_f$ is almost one-to-one, it is strongly proximal.
\end{proof}


\end{document}